\documentclass{amsart}

\usepackage[mathscr]{eucal}
\usepackage{amsfonts}
\usepackage{amsmath, amssymb}
\usepackage{amsthm}
\usepackage{aligned-overset}
\usepackage{comment}

\makeatletter
\def\@setsubjclass{%
  {\itshape 2020 Mathematics Subject Classification.}\enspace
  \@subjclass\@addpunct.}
\makeatother

\numberwithin{equation}{section}
\DeclareMathAlphabet\mathcal{OMS}{cmsy}{m}{n}
\SetMathAlphabet\mathcal{bold}{OMS}{cmsy}{b}{n}

\newtheorem{theorem}{Theorem}[section]
\newtheorem{lemma}[theorem]{Lemma}

\newtheorem{definition}[theorem]{Definition}

\newtheorem{remark}[theorem]{Remark}

\newcommand{\mR}{\mathbb{R}}
\newcommand{\mD}{\mathbb{D}}
\newcommand{\mP}{\mathbb{P}}
\newcommand{\mN}{\mathbb{N}}
\newcommand{\mL}{\mathbb{L}}

\newcommand{\shs}[1]{\hspace*{#1cm}}

\newcommand{\cb}{\mathcal{B}}
\newcommand{\cc}{\mathcal{C}}
\newcommand{\ce}{\mathcal{E}}
\newcommand{\cf}{\mathcal{F}}
\newcommand{\cg}{\mathcal{G}}
\newcommand{\ch}{\mathcal{H}}
\newcommand{\cn}{\mathcal{N}}
\newcommand{\cs}{\mathcal{S}}

\begin{document}

\title[Absolute continuity for the fractional Cahn-Hilliard equation] {Absolute continuity of the law of the solution to the stochastic Cahn-Hilliard equation driven by fractional noise}

\author[Dimitriou, Farazakis, Karali]{Dimitrios Dimitriou$^{\dag\ddagger}$, Dimitris Farazakis$^{**\ddagger}$, Georgia Karali$^{*\ddagger}$}

.

\thanks
{$^{\dag}$ Department of Mathematics and Applied Mathematics, University of Crete, Heraklion, Greece.}
\thanks
{$^{**}$ Department of International and European Economic Studies, School of Economics, University of Western Macedonia, \hspace*{0.85cm}Kozani, Greece.}
\thanks
{$^{*}$ Department of Mathematics, National and Kapodistrian University of Athens, Panepistimiopolis, Athens, Greece.}
\thanks
{$^{\ddagger}$ Institute of Applied and Computational Mathematics, FORTH, Heraklion, Greece.}
\thanks{E-mails: mathp429@math.uoc.gr, dfarazakis@uowm.gr, gkarali@math.uoa.gr}

\begin{abstract}
The stochastic partial differential equation analyzed in this work is the Cahn-Hilliard equation perturbed by an additive fractional white noise (fractional in time and white in space). We work in spatial dimensions $d \in \{1, 2, 3\}$ with $d < 4H$, where $H \in (\frac{1}{2}, 1)$ is the Hurst parameter of the noise, and apply Malliavin calculus to investigate the existence of a density for the solution $u$. In particular, we show that $u$ admits a jointly continuous modification and construct a localizing sequence through which we prove that its Malliavin derivative exists locally, and that its law is absolutely continuous with respect to the Lebesgue measure on $\mR$, thereby establishing that a density exists. We obtain the result for the whole range of dimensions and Hurst parameters in which the mild solution is known to exist. The analysis rests on three new estimates. The first bounds the expectation of the $p$-th power ($p \geq 2$) of the $L^{\infty}(D)$-norm of the stochastic integral appearing in the mild formulation. The other two concern the integral involving the operator associated with the kernel that appears in the integral representation of the fractional noise: an upper estimate and, more importantly, a lower estimate on short time intervals, which yields the nondegeneracy of the Malliavin derivative.
\end{abstract}

\maketitle

\pagestyle{myheadings}

\thispagestyle{plain}

\section{Introduction} \label{Section 1}
\subsection{The stochastic model} \label{Subsection 1.1}
Let $T>0$, $d \in \{1,2,3\}$ and $D:= [0,\pi]^d$. In this work we study the following stochastic version of the Cahn-Hilliard equation perturbed by an additive fractional white noise $\dot{W}_{H}$
\begin{equation} \label{stoch. model}
u_{t} = - \Delta \big{(} \Delta u - f(u) \big{)} + \sigma \dot{W}_H\quad \text{in}\ D\times(0,T].
\end{equation}
Here, $\sigma \in \mR_*$ and $f$ is a polynomial function of degree three with constant coefficients (independent of $t$ and $x$) and a positive dominant coefficient. In addition, the fractional white noise $\dot{W}_{H}$ (the proper definition will be given in Subsection \ref{Subsection 2.2}) is understood in the sense of Walsh \cite{W}, i.e., as the formal derivative $\frac{\partial^{d + 1} W_{H}}{\partial x_1 \dotsm \partial x_d \partial t}$, in the sense of (Schwartz) distributions, of the continuous version of the zero mean Gaussian process
$$\big{\{}W_{H}(x,t)\, :\, (x,t) \in D\times[0,T]\big{\}},\quad  H \in \big{(} \tfrac{1}{2}, 1 \big{)},$$
with covariance
$$E \big{(} W_{H}(x,t) W_{H}(y,s) \big{)} = \bigg{(} \prod_{i = 1}^d \min\{x_i, y_i\} \bigg{)} \cdot \frac{t^{2H} + s^{2H} - |t-s|^{2H}}{2}.$$
The process $W_{H}$ is defined on a complete probability space $(\Omega,\cf,\mP)$ endowed with the standard filtration $\{\cf^{}_{t}\}_{t\in[0,T]}$ of the process, that is, the standard enlargement (satisfying the usual conditions: right-continuity and completeness) of the filtration generated by $W_{H}$, i.e.,
$\sigma \big{(} \big{\{} W_{H}(x,s)\, :\, x\in D,\ s\leq t \big{\}} \big{)},\ t\in[0,T].$
Note that $W_{H}$ is a Brownian sheet in the space variables and a fractional Brownian motion with Hurst parameter $H \in \big{(} \frac{1}{2}, 1 \big{)}$ in the time variable.

The initial and boundary value problem associated with (\ref{stoch. model}) that we will study comprises an initial condition 
$u(x,0) = u_{0}(x)$ in $D$
together with the homogeneous Neumann boundary conditions
\begin{equation} \label{Neumann b.c.}
\dfrac{\partial u}{\partial n} = \dfrac{\partial (\Delta u)}{\partial n} = 0\quad \text{on}\ \partial D \times [0,T].
\end{equation}

The Cahn-Hilliard equation describes the phenomenon of spinodal decomposition in metal alloys. In particular, it describes the mechanism by which the two components of a binary fluid spontaneously separate and form regions that are pure in each constituent \cite{CH}. Stochasticity was incorporated into the deterministic partial differential equation in \cite{C} with the introduction of thermal fluctuations in the form of an additive space-time white noise. Various versions of this stochastic Cahn-Hilliard equation have been thoroughly studied and many interesting results regarding the existence of a solution and its density, the uniqueness and the regularity of the solution, among others, have been derived. For instance, in \cite{WEB1} all the aforementioned topics were analyzed in the case of a multiplicative space-time white noise. In \cite{AKM}, a combined Cahn-Hilliard/Allen-Cahn model was considered, where the authors investigated the existence and regularity of the solution with unbounded noise diffusion in dimensions one, two and three.

Stochastic partial differential equations driven by additive or multiplicative fractional noises have also been the subject of much research. This stems from the self-similar and long-range dependence properties of the fractional Brownian motion, which make it a suitable candidate to describe data in fields like physics, biology, hydrology, network traffic modeling and financial mathematics (see, e.g., \cite{LTWW,M}). Some notable examples regarding the theoretical analysis of equations of this type are \cite{NO}, where the authors proved existence and uniqueness of mild solutions to a class of second-order heat equations with additive fractional noise, \cite{H}, where Lyapunov exponent estimates on the solutions of second-order Anderson models with multiplicative fractional noise potentials were derived, and \cite{NS}, where the differentiability in the direction of the Cameron-Martin space and the existence of a density for the solution of a stochastic differential equation driven by fractional Brownian motion were proved. For a thorough study of Malliavin calculus and stochastic partial differential equations, we refer the reader to the books \cite{N} and \cite{dalang}, respectively.

We focus on the stochastic Cahn-Hilliard equation driven by additive fractional white noise. In \cite{BJW}, the existence of a unique (global) mild solution was established in spatial dimensions $d<4H$, where $H \in \big{(}\frac{1}{2},1 \big{)}$ is the Hurst parameter of the fractional noise, via a weak convergence argument. We aim to establish the
existence of a density for the mild solution in the same range of dimensions. Our approach introduces new estimates that extend the analysis of \cite{AFK1, WEB1} beyond the space-time white noise setting. 
\subsection{Main results} \label{Subsection 1.2}
We investigate whether the random variable $u(x,t)$, the unique solution of (\ref{stoch. model}) at fixed $(x,t)$, has a density; an affirmative answer is given by proving that its law is absolutely continuous with respect to the Lebesgue measure on $\mR$. Specifically, we present a suitable sequence $\{\Omega_{n}\}_{n \in \mN_{*}} \subset \cf$ and construct an approximating sequence $\{u_{n}(x,t)\}_{n \in \mN_{*}}$ of $u(x,t)$ for which we prove existence of the Malliavin derivative and the almost surely strict positivity of a particular norm involving it. This allows us to employ important results from the theory of Malliavin calculus, presented in \cite{N}, to show the absolute continuity of the law of $u_{n}(x,t)$ with respect to the Lebesgue measure on $\mR$. The sequence $\{(\Omega_{n},u_{n}(x,t))\}_{n \in \mN_{*}}$ localizes $u(x,t)$ in $\mD^{1,2}$ (we refer to Subsection \ref{Subsection 2.1}), thereby allowing both of these properties to be transferred from $u_{n}(x,t)$ to $u(x,t)$, which ultimately yields our main result.

The localization architecture is the one introduced in \cite{WEB1} and reused in \cite{AFK1}: the sets (\ref{def. of Omega_n}) and the truncation $f_n$ are those of \cite[Section 3.1]{WEB1}, while Lemmas \ref{main lemma 1} and \ref{main lemma 2} share the structure of the argument with \cite[Lemma 3.2]{WEB1}, as does Subsection \ref{Subsection 4.1} with \cite[Subsection 3.2]{WEB1}. We record this explicitly in order to isolate the places at which that scheme breaks down under a fractional noise. We describe them in turn, and then the new estimates that overcome them, which are the key analytic ingredients allowing us to establish absolute continuity throughout the full existence range $d < 4H$.

The first obstacle concerns a moment bound in $L^{p} \big{(} \Omega; L^{\infty}(D) \big{)}$ for the stochastic integral appearing in (\ref{mild formulation}). The sup-norm cannot be dispensed with here. Theorem \ref{main theorem} concerns the law of $u(x, t)$ at an arbitrary fixed point and therefore requires $u_n(x,t) \in \mD^{1,2}$ at every such point; a scheme run in $L^q(D)$ with $q < \infty$ (which is what the convolution estimates of \cite[Lemma 3.3]{BJW} are designed for) controls $u_n(\cdot, t)$ as an element of $L^q(D)$, and hence yields the required membership only for almost every $x$. In \cite{WEB1}, such a bound stems from Burkholder's inequality, the space-time $L^{2}$-estimates for the Green's function $G$ and its increments, and a Garsia-Rodemich-Rumsey argument. In our case, the fractional noise forces us to obtain $\ch$-estimates instead, where $\ch$ is the Hilbert space canonically associated with it, whose inner product (\ref{S.3}) carries two time arguments coupled through the singular kernel $|s - r|^{2H - 2}$. The estimates of \cite[Lemma 1.8]{WEB1} are the single time integrals produced by the It\^o isometry and thus control $G$ and its increments only on the diagonal $s = r$, whereas (\ref{S.3}) calls for the $L^{2}(D)$-norm of the integrand at each fixed time. We therefore establish the pointwise-in-time bounds of Lemma \ref{Green's function's properties} and compute all $\ch$-norms directly from (\ref{S.3}).

The second obstacle concerns the nondegeneracy of the Malliavin derivative of $u_n(x,t)$. In \cite{AFK1} and \cite{WEB1} the noise diffusion is first removed from the problem (in the former via the assumption that it is bounded away from zero, in the latter by factoring it out of the Malliavin derivative under the weaker assumption that it does not vanish) and the property then follows from a lower bound on the space-time integral of $G^2$. The bound is obtained from the series expansion of 
$G$: the orthonormality of the eigenfunctions collapses the space integration and leaves a series whose terms are exponential in time, so that they can be integrated one by one. In our case $G$ is first integrated in time against $\frac{\partial K_H}{\partial r}$, where $K_H$, given by (\ref{S.1}), is the kernel of the integral representation of the noise; squaring then leaves a double time integral carrying two copies of that derivative. If we were to take the same route, each term of the resulting series would be an exponential integrated against two copies of a singular kernel, and the term-by-term evaluation on which \cite{WEB1} rests would no longer be available.

Three new estimates resolve these difficulties. Estimate (\ref{stoch estimate}) bounds the expectation of the $p$-th power ($p \geq 2$) of the $L^{\infty}(D)$-norm of the stochastic integral appearing in (\ref{mild formulation}); the additivity of the noise renders the integral in question Gaussian, so that its moments reduce to powers of the $\ch$-norm of its integrand without any appeal to Burkholder's inequality. The same reduction, applied to the increments of that integrand in each of its two variables, yields in addition the joint H\"older continuity of that integral, of any order below $H - \frac{d}{8}$; this provides the jointly continuous modification of the stochastic term used in Lemma \ref{continuity of u}, so that no external regularity estimate is needed for it. Estimate (\ref{sup estimate}) bounds the deterministic integral built from the associated operator $K_H^*$ and $G$, uniformly in the space variable. Its exponent, $2H - \frac{d}{4}$, is what ultimately renders the drift contribution negligible against that of the noise in the limit taken in Lemma \ref{abs. con. of law of u_n}, and it does so for every $d < 4$, hence throughout the range in which the mild solution exists. Estimate (\ref{lower bound}), lastly, is the one which yields the nondegeneracy of the Malliavin derivative; the aforementioned issue is bypassed by combining the semigroup property of $G$ with a pointwise lower bound for $G$ on the diagonal and a lower bound for the kernel $K_H$ itself. Its constant depends only on $\sigma$, $d$ and $H$, so that the estimate holds at every point of $D$, in contrast with \cite{AFK1} and \cite{WEB1} where the density result is confined to the interior of the domain. \\[0.2cm]
\indent In what follows, we state the main result of this paper and outline the rest of its structure.

\begin{theorem} \label{main theorem}
Let $u$ be the unique solution of $(\ref{stoch. model})$, with $H \in \big{(} \frac{1}{2},1 \big{)}$ and $d < 4H$, subject to the homogeneous Neumann boundary conditions $(\ref{Neumann b.c.})$ and with a deterministic initial condition $u_{0} \in C(D)$. Then, the Malliavin derivative of $u(x,t)$ exists locally for every $(x,t) \in D \times [0,T]$. Furthermore, for any $(x,t) \in D \times (0,T]$, the law of the random variable $u(x,t)$ is absolutely continuous
with respect to the Lebesgue measure on $\mR$.
\end{theorem}
In Section \ref{Section 2}, we present some basic definitions from the theory of Malliavin calculus such as the notion of the Malliavin derivative, the spaces of random variables and stochastic processes $\mD^{1,2}$, $\mL^{1,2}$, and their local versions $\mD_{\text{loc}}^{1,2}$, $\mL_{\text{loc}}^{1,2}$. Furthermore, we present the definition and some properties of our fractional noise as well as the weak and mild formulations of the solution of $(\ref{stoch. model})$. 

In Section \ref{Section 3}, we first collect a number of auxiliary estimates for the Green's function (Lemma \ref{Green's function's properties}) and use them to derive important estimates regarding the operator $K_H^*$ acting on the Green's function and the stochastic integral appearing in the mild formulation of the solution $u$ of $(\ref{stoch. model})$, as well as the joint H\"older continuity of that integral (Lemma \ref{Estimates for stoch. int.}); we then prove the existence of a jointly continuous modification of $u$ (Lemma \ref{continuity of u}). Moreover, we show that $u(x,t) \in \mD_{\text{loc}}^{1,2}$ by constructing a suitable localizing sequence $\{(\Omega_{n},u_{n}(x,t))\}_{n \in \mN_{*}}$; this construction involves the existence and uniqueness, up to indistinguishability, of an approximating sequence $\{u_{n}\}_{n \in \mN_{*}}$, which admits a jointly continuous modification and coincides with $u$ on $\Omega_{n}$ (Lemma \ref{main lemma 1}), and the representation of the Malliavin derivative of $u_n(x,t)$ (Lemma \ref{main lemma 2}) as the solution to a random linear integral equation. The Malliavin derivative of $u(x,t)$ is then directly determined by that of $u_{n}(x,t)$ (see Remark \ref{important remark}).

In Section \ref{Section 4}, we present an important estimate involving the Malliavin derivative of $u_{n}(x,t)$ (Lemma \ref{important estimate}) that is then used to prove the absolute continuity of the law of $u_{n}(x,t)$ with respect to the Lebesgue measure on $\mR$ (Lemma \ref{abs. con. of law of u_n}). Lastly, we transfer the nondegeneracy of the Malliavin derivative from $u_{n}(x,t)$ to $u(x,t)$ and deduce the absolute continuity of the law of $u(x,t)$ (Lemma \ref{abs. con. of law of u}), which in turn implies the existence of a density for $u(x,t)$, thereby completing the proof of Theorem \ref{main theorem}.

Throughout this work, $T > 0$ is fixed.
As usual, the generic positive constant, depending on the parameter $T$ for example, will be denoted by $c(T)$ or simply $c$ if the dependence is not essential. Its exact value may also change from line to line. In many instances throughout this work, the symbols $\cdot, \diamond, \bullet$ will serve as placeholders for the first space argument, the second space argument and the time argument, respectively.
\section{Preliminaries} \label{Section 2} 
\subsection{Basic concepts from Malliavin calculus} \label{Subsection 2.1}
$ $\newline \indent
The definitions we will present are based on the work of Nualart \cite{N}. In the book, the author provides a general framework regarding the Malliavin derivative through the use of an isonormal Gaussian process $\{W(h)\}_{h\in \mathscr{H}}$ where $\mathscr{H}$ is a real separable Hilbert space. Our work and the definitions that we will provide correspond to the special case $\mathscr{H} = L^{2} ( D \times [0,T] )$.
\begin{definition} \label{Def. of i.G.p.}
A stochastic process $\{W(h)\}_{h\in \mathscr{H}}$ defined on a complete probability space $(\Omega,\cf,\mP)$ is called an isonormal Gaussian process if $W$ is a centered Gaussian family of random variables such that 
$$E \big{(} W(h) W(g) \big{)} = \langle h,g \rangle_{\mathscr{H}},\ \forall h,g \in \mathscr{H}.$$
\end{definition}
\noindent It is easy to see that the mapping $h \mapsto W(h)$ is linear in $h$, almost surely. In fact, this mapping provides a linear isometry of $\mathscr{H}$ onto a closed subspace of $L^{2}(\Omega)$.

We now fix an isonormal Gaussian process $\{W(h)\}_{h\in \mathscr{H}}$ defined on a complete probability space $(\Omega,\cf,\mP)$ where $\cf = \sigma (\{ W(h)\, :\, h \in \mathscr{H} \})$. We denote by $\cs$ the set of random variables of the form 
\begin{equation} \label{elements of S}
F = \varphi \big{(} W(h_{1}), \dotsc, W(h_{n}) \big{)},
\end{equation}
where $\varphi \in C_{p}^{\infty}(\mR^{n})$ (the set of all infinitely differentiable functions such that $\varphi$ and all of its partial derivatives have polynomial growth), $h_{1}, \dotsc, h_{n} \in \mathscr{H}$ and $n \in \mN_{*}$. The $($Malliavin$)$ derivative of a random variable $F\in\cs$ is then defined as the $\mathscr{H}$-valued random variable
$$\Tilde{D}F := \sum\limits_{i=1}^{n} \frac{\partial \varphi}{\partial x_{i}}\big{(}W(h_{1}),\dotsc,W(h_{n})\big{)}h_{i}.$$
We note that the operator $\Tilde{D}:\cs\to L^{2}(\Omega;\mathscr{H})$ is linear and unbounded. Moreover, it can be shown that it is closable from $L^{2}(\Omega)$ to $L^{2}(\Omega;\mathscr{H})$ (see \cite[Proposition 1.2.1]{N}); thus it has a closed extension. We denote this closed extension by $D$ and its domain by $\mD^{1,2}$. More specifically, we have the following definition:
\begin{definition}
The space $\mD^{1,2}$ is defined as the closure of $\cs$ with respect to the norm 
$$\|F\|_{\mD^{1,2}}^{2}:=\|F\|^{2}_{L^{2}(\Omega)}+\|DF\|_{L^{2}(\Omega;\mathscr{H})}^{2}.$$
\end{definition}
\begin{remark} \label{R.R. 11}
$1)$ $\mD^{1,2}$ is a Hilbert space with inner product
$\langle F,G\rangle_{\mD^{1,2}}:=\langle F,G\rangle_{L^{2}(\Omega)}+\langle DF,DG\rangle_{L^{2}(\Omega;\mathscr{H})}$. 
\\[0.1cm]
$2)$ Given a random variable $F\in\mD^{1,2}$, the Malliavin derivative $DF$ is an element of $L^{2}(\Omega;\mathscr{H})$ which in our case can be identified with $L^{2}(\Omega \times D\times[0,T])$. Thus the Malliavin derivative can be viewed as a stochastic process $\{D_{y,s}F\}_{(y,s)\in D\times[0,T]}$ where $D_{y,s}F$ is defined almost everywhere with respect to the measure $\mP\otimes m$ $(m$ being the Lebesgue measure$)$. 
\\[0.1cm]
$3)$ The above definition can be extended to Hilbert-space-valued random variables. In particular, if $\mathscr{V}$ is any real separable Hilbert space, then the space $\mD^{1,2}(\mathscr{V})$ is defined as the completion of the space $\cs_{\mathscr{V}}$ consisting of random variables of the form $F = F_1 v_1 + \dotsc + F_n v_n,\ F_i \in \cs,\ v_i \in \mathscr{V},\ i \in \{1, \dotsc, n\}$
with respect to $$\|F\|_{\mD^{1,2}(\mathscr{V})}^2 := \|F\|_{L^{2}(\Omega ; \mathscr{V})}^2 + \|DF\|_{L^{2}(\Omega ; \mathscr{H} \otimes \mathscr{V})}^2 \quad \text{where} \quad DF := DF_1 \otimes v_1 + \dotsc + DF_n \otimes v_n.$$
Of particular interest to us is the space $\mD^{1,2}(\mathscr{V})$ in the special case where $\mathscr{V} = \mathscr{H} = L^{2} ( D \times [0,T] )$. This is a Hilbert space, isomorphic to $L^{2} \big{(} D \times [0,T];\mD^{1,2} \big{)}$, and is denoted by $\mL^{1,2}$ with norm $$\|F\|_{\mL^{1,2}}^2 = \|F\|_{L^2(\Omega \times D \times [0,T])}^2 + \|DF\|^2_{L^2(\Omega \times (D \times [0,T])^2)}.$$
\end{remark}
\begin{definition} \label{Def. of local versions}
The local versions of the spaces $\mD^{1,2},\ \mL^{1,2}$ are defined as follows:
\begin{align*}
\mD^{1,2}_{\textup{\text{loc}}} &:= \big{\{}\textup{\text{r.v.}}\ F\, :\, \exists\, \{(\Omega_{n},F_{n})\}_{n\in\mN_{*}} \subset \mathcal{F}\times \mD^{1,2}\ \textup{\text{such that}}\ \Omega_{n}\uparrow \Omega\ \textup{\text{a.s.}}\ \&\ F=F_{n}\ \textup{\text{a.s. on}}\ \Omega_{n}\big{\}}, \\
\mL^{1,2}_{\textup{\text{loc}}} &:= \big{\{}\textup{\text{st.p.}}\ F\, :\, \exists\, \{(\Omega_{n},F_{n})\}_{n\in\mN_{*}} \subset \mathcal{F}\times \mL^{1,2}\ \textup{\text{such that}}\ \Omega_{n}\uparrow \Omega\ \textup{\text{a.s.}}\ \&\ F=F_{n}\ \textup{\text{a.s. on}}\ \Omega_{n}\big{\}},
\end{align*}
where
$$\Omega_{n}\uparrow \Omega\ \textup{\text{a.s.}}\iff\Omega_{1}\subseteq\Omega_{2}\subseteq  \cdots\subseteq\Omega\quad \textup{\text{such that}}\quad \underset{n\to+\infty}{\lim}\mP(\Omega_{n})=\mP(\Omega)=1.$$
\end{definition}
\begin{remark} \label{important remark} 
$1)$ It is straightforward that   $\mD^{1,2}\subseteq\mD^{1,2}_{\textup{\text{loc}}}$ and $\mL^{1,2}\subseteq\mL^{1,2}_{\textup{\text{loc}}}$. Moreover, it holds that $\mD^{1,2}\subseteq\mD^{1,1}$ (see \cite[p. 27]{N}) thus $\mD^{1,2}_{\textup{\text{loc}}} \subseteq \mD^{1,1}_{\textup{\text{loc}}}$. 
\\[0.1cm]
$2)$ If $F \in \mD^{1,2}_{\textup{loc}}$ and $\{(\Omega_n, F_n)\}_{n \in \mN_{*}}$ localizes $F$ in $\mD^{1,2}$, then $DF$ is defined without ambiguity by $DF = DF_n$ on $\Omega_n,\ n \geq 1$.
\end{remark}
\subsection{Fractional white noise and related properties} \label{Subsection 2.2}
$ $\newline
The definitions we will present are based on the work of Nualart and Ouknine \cite{NO}. In what follows, $\cb,m$ denote the Borel $\sigma$-algebra and the Lebesgue measure respectively.
\begin{definition}
Let $D \subseteq \mR^d$ and $H \in (0,1)$. A centered Gaussian random field 
$$\big{\{}W_{H}(U\times[0,t])\, :\, U\in\cb(D),t\in[0,T]\big{\}},$$
defined on a complete probability space $(\Omega,\cf,\mP)$, with covariance function
\begin{equation*}
E \big{(} W_{H}(U\times[0,t]) W_{H}(V\times[0,s]) \big{)} = m(U\cap V) \cdot \frac{t^{2H} + s^{2H} - |t-s|^{2H}}{2},\ \forall U,V\in\cb(D), \forall t,s\in[0,T],
\end{equation*}
is called a fractional white noise (fractional in time, white in space) on $D\times[0,T]$ of Hurst parameter $H$.
\end{definition}
We denote now by $\ce$ the set of step functions on $D\times[0,T]$ and define the Hilbert space $\ch$ as the closure of $\ce$ with respect to the scalar product 
\begin{equation*}
\big{\langle} \mathbf{1}_{U \times [0,t]}, \mathbf{1}_{V \times [0,s]}  \big{\rangle}_{\mathcal{H}} :=
E \big{(} W_{H}(U \times [0,t]) W_{H}(V \times [0,s]) \big{)}.
\end{equation*}
The mapping $\mathbf{1}_{U \times [0,t]} \mapsto W_{H}(U \times [0,t])$ can be extended to an isometry between $\ch$ and a closed subspace of $L^{2}(\Omega,\cf,\mP)$. We will denote this isometry by $\hat{W}_{H}(\phi)$ and, when $\phi$ is a function, use the notation 
\begin{equation} \label{fractional i.G.p}
\int_{0}^{T}\int_{D}\phi(y,s)W_{H}(dy,ds) := \hat{W}_{H}(\phi).
\end{equation}
\begin{remark} \label{S.2}
$1)$ $\{\hat{W}_{H}(\phi)\}_{\phi\in\ch}$ is an isonormal Gaussian process in the sense of Definition \ref{Def. of i.G.p.}. \\
$2)$ For any $H \in \big{(} \frac{1}{2}, 1 \big{)}$ it holds that $L^{2}(D \times [0,T]) \subseteq \ch$,
and for all $\phi, \psi\in L^{2}(D\times[0,T])$
\begin{equation} \label{S.3}
\langle \phi, \psi \rangle_{\ch}
= H(2H - 1) \int_{D} \int_0^T \int_0^T
\phi(y, s)\, \psi(y, r)\, |s - r|^{2H - 2}\, ds dr dy.
\end{equation}
See \cite[Subsection 2.1]{BaTu}, where both statements are proved on $\mR^d$; the argument applies verbatim on $D$.
\end{remark}

We know that the fractional Brownian motion is a Volterra process with covariance $R_{H}(t,s)$, which can be written as (see \cite[p. 278]{N})
\begin{equation*}
\frac{t^{2H}+s^{2H}-|t-s|^{2H}}{2}=: R_{H}(t,s) = \int_{0}^{\min\{t,s\}}K_{H}(t,\tau)K_{H}(s,\tau)\, d\tau,
\end{equation*}
where, in our case of $H \in \left (\frac{1}{2},1 \right )$, the kernel $K_{H}$ is given by
\begin{align} \label{S.1}
\begin{split}
K_{H}(t,s) :&= \kappa_{1}\, s^{\frac{1}{2}-H} \int_{s}^{t} (u-s)^{H - \frac{3}{2}} u^{H - \frac{1}{2}}\, du\\
&= \kappa_{2} \bigg{[} (t-s)^{H-\frac{1}{2}} + \Big{(} H - \frac{1}{2} \Big{)} \int_{s}^{t} (u-s)^{H-\frac{3}{2}} \bigg{(} \Big{(} \frac{u}{s} \Big{)}^{H - \frac{1}{2}} - 1 \bigg{)}  du \bigg{]}
\end{split}
\end{align}
for $0 < s < t \leq T$ and 
$$\kappa_1 := \bigg{(} \dfrac{H(2H - 1)}{\text{B}(2 - 2H, H - \frac{1}{2})} \bigg{)}^{\frac{1}{2}},\quad \kappa_2 := \bigg{(} \dfrac{2H \Gamma(\frac{3}{2} - H)}{\Gamma(H + \frac{1}{2}) \Gamma(2 - 2H)} \bigg{)}^{\frac{1}{2}}$$
with $\text{B},\Gamma$ being
the beta and gamma functions respectively. It is also straightforward to verify that 
\begin{equation} \label{par. der. of K_H}
\dfrac{\partial K_{H}}{\partial t}(t,s) = \kappa_{1} \Big{(} \frac{t}{s} \Big{)}^{H - \frac{1}{2}} (t - s)^{H-\frac{3}{2}} > 0,
\end{equation}
\begin{equation} \label{int. of par. der. of K_H}
\int_{s}^{t} \frac{\partial K_{H}}{\partial r}(r,s)\, dr = K_{H}(t,s),
\end{equation}
for every $0 <s < t \leq T$, where (\ref{int. of par. der. of K_H}) follows from (\ref{par. der. of K_H}), together with the fact that $K_{H}(r,s) \to 0$ as $r \downarrow s$ by (\ref{S.1}), via the fundamental theorem of calculus on $[\rho,t]$, $s < \rho < t$, and monotone convergence as $\rho \downarrow s$.

We consider now the linear operator $K_{H}^{*}:\ce\to L^{2}(D\times[0,T])$ defined by
$$\big{[} K_{H}^{*} ( \phi ) \big{]}(x,s):=\int_{s}^{T}\phi(x,r)\frac{\partial K_{H}}{\partial r}(r,s)\, dr,$$
and observe that
\begin{equation*} 
\big{[} K_{H}^{*} (\mathbf{1}_{U\times[0,t]} ) \big{]}(x,s) = K_{H}(t,s) \mathbf{1}_{U\times[0,t]}(x,s),\ \forall U\in\cb(D), \forall t\in[0,T].
\end{equation*}
As a consequence, the operator $K_{H}^{*}$ is an isometry between $\ce$ and $L^{2}(D\times[0,T])$ that can be extended to $\ch$.\\
In addition, $K_{H}^{*}$ is surjective (see \cite[Lemma 2.3]{BaTu}). These properties imply that $K_{H}^{*}$ is invertible and thus the Gaussian random field
\begin{equation} \label{space-time w.n.}
W(U\times[0,t]) := \hat{W}_{H} \big{(} (K_{H}^{*} )^{-1}  ( \mathbf{1}_{U\times[0,t]} ) \big{)},\ U\in\cb(D),\  t\in[0,T]
\end{equation}
is well-defined and is in fact a space-time white noise. Furthermore, it holds almost surely
$$W_{H}(U\times[0,t]) = \displaystyle \int_{0}^{t} \int_{U} K_{H}(t,s) W(dy,ds),\ \forall U\in\cb(D),\forall t\in[0,T]$$
and
\begin{equation} \label{int. repres. of fractional i.G.p.}
\hat{W}_{H}(\phi) = \displaystyle \int_{0}^{T} \int_{D} \big{[} K^{*}_{H} (\phi) \big{]} (y,s) W(dy,ds),\ \forall \phi \in \ch.
\end{equation} 
\begin{remark} \label{R.R. 12}
The filtrations generated by $W$ and $W_{H}$ are the same (see \cite[pp. 65-66]{BaTu}).
\end{remark} 
\subsection{Weak and mild solutions} \label{Subsection 2.3}
$ $\newline  \indent
Let $\big{(} \Omega,\cf,\{\cf_{t}\}_{t\in[0,T]},\mP \big{)}$ be the complete, filtered probability space mentioned in Subsection \ref{Subsection 1.1}.
\begin{definition}
A jointly measurable, $\cf_{t}$-adapted process $u = \{u(x,t)\}_{(x,t) \in D\times[0,T]}$ is called a weak solution of $(\ref{stoch. model})$ subject to the homogeneous Neumann boundary conditions $(\ref{Neumann b.c.})$ and with a deterministic initial condition $u_{0} \in C(D)$, if it satisfies the following weak formulation

\begin{align} \label{weak formulation}
\begin{split}
\int_{D} \big{(} u(x,t) - u_{0}(x) \big{)} \phi(x)\, dx = \int_{0}^{t} \int_{D} -\Delta^2 \phi(x) u(x,s) &+ \Delta \phi(x) f\big{(} u(x,s) \big{)}\, dxds \\
&+ \sigma \int_{0}^{t} \int_{D} \phi(x) W_{H}(dx,ds)
\end{split}
\end{align}
almost surely, for all $\phi\in C^{4}(D)$ with $\dfrac{\partial \phi}{\partial n} = \dfrac{\partial (\Delta \phi)}{\partial n} = 0$ on $\partial D$ and all $t \in [0,T]$.
\end{definition}
\begin{remark} \label{well defined stoch. int.}
The last term of $(\ref{weak formulation})$ is well-defined. Indeed, (we recall that we work for $H \in \big{(} \frac{1}{2},1 \big{)}$)
$$\int_{0}^{t} \int_{D} \phi(x) W_{H}(dx,ds) = \int_{0}^{T} \int_{D} \phi(x) \mathbf{1}_{[0,t]}(s) W_{H}(dx,ds).$$
If we define $\psi(x,s) := \phi(x) \mathbf{1}_{[0,t]}(s)$, then $\psi \in L^{2} \big{(} D \times [0,T] \big{)} \subseteq \ch$ and thus
$$ \int_{0}^{T} \int_{D} \phi(x) \mathbf{1}_{[0,t]}(s) W_{H}(dx,ds) \overset{(\ref{fractional i.G.p})}{=} \hat{W}_{H}(\psi) \overset{(\ref{int. repres. of fractional i.G.p.})}{=} \int_{0}^{T} \int_{D} \big{[} K^{*}_{H} ( \psi ) \big{]} (x,s) W(dx,ds).$$
\end{remark}
If we consider now the Green's function for the operator $\frac{\partial}{\partial t} + \Delta^2$ with homogeneous Neumann boundary conditions, we can then present the mild formulation of (\ref{stoch. model}). More specifically, $u$ is a solution of (\ref{weak formulation}) if and only if it satisfies the following equation almost surely for any $x\in D$ and any $t\in[0,T]$:
\begin{equation} \label{mild formulation}
\begin{split}
    u(x,t) =\int_{D} G(x,y,t) u_{0}(y)\, dy &+ \int_{0}^{t} \int_{D} \Delta_y G(x,y,t-s) f\big{(} u(y,s) \big{)}\, dyds \\ &+ \sigma \int_{0}^{t} \int_{D} G(x,y,t-s) W_{H} (dy,ds),
\end{split}
\end{equation}
where
\begin{equation} \label{Green's function}
G(x,y,t) = \sum\limits_{k \in \mN^d}e^{-|k|^{4} t} a_{k}(x) a_{k}(y)\quad \&\quad a_k(x) := \prod_{i = 1}^d \theta_{k_i}(x_i),\quad x,y\in D,\ t>0
\end{equation} 
with $|k|^2 = \sum\limits_{i = 1}^d k_i^2$ and $\theta_0(\xi):= \frac{1}{\sqrt{\pi}}$, $\theta_j(\xi) := \sqrt{\frac{2}{\pi}} \cos(j \xi)$ for $j \neq 0$; in particular, $a_0 = \pi^{-\frac{d}{2}}$ on $D$.
\begin{remark}
$1)$ It is easy to verify, using (\ref{Green's function's estimates}) $i)$ and (\ref{specific calculation}), that $G(x,\diamond,t - \bullet) \mathbf{1}_{[0,t]}(\bullet) \in L^{2}(D \times [0,T])$ for any $(x, t) \in D \times [0,T]$. Thus, similarly to Remark \ref{well defined stoch. int.}, the last term of (\ref{mild formulation}) is well-defined and equals almost surely
\begin{equation} \label{formula for the stoch. int.}
\int_0^t \int_D G(x, y, t - s) W_{H} (dy, ds) = \int_0^T \int_D \big{[} K^{*}_{H} \big{(} G(x, \diamond, t - \bullet) \mathbf{1}_{[0,t]}(\bullet) \big{)} \big{]} (y, s) W(dy, ds).
\end{equation}
Furthermore, the integral on the right-hand side of (\ref{formula for the stoch. int.}) is in fact taken over $D \times [0,t]$; combined with Remark \ref{R.R. 12}, this shows that the stochastic term of (\ref{mild formulation}) is $\cf_{t}$-measurable. \\
$2)$ The expansion (\ref{Green's function}) is symmetric, i.e. $G(x, y, t) = G(y, x, t)$, so by the orthonormality of $\{a_{k}\}_{k \in \mN^d}$ in $L^{2}(D)$ we have, for any $x, y \in D$ and any $t_1, t_2 > 0$,
\begin{equation} \label{int. property of Green's function}
\int_D G(x, z, t_1)\, G(z, y, t_2)\, dz = \int_D G(x, z, t_1)\, G(y, z, t_2)\, dz
= G(x, y, t_1 + t_2),
\end{equation}
that is, the space integral of a product of two Green's functions collapses onto a single one
evaluated at the sum of the two times. On the diagonal $y = x$ the terms of (\ref{Green's function}) are non-negative; retaining only the one indexed by $k = 0$, whose value is $a_0^2 = \pi^{-d}$, gives the first bound below; the second follows from (\ref{Green's function's estimates}) $i)$ taken at $y = x$:
\begin{equation} \label{diagonal bounds}
\pi^{-d} \leq G(x, x, t),\ \forall (x,t) \in D \times (0, +\infty)\quad \&\quad
G(x, x, t) \leq c\, t^{-\frac{d}{4}},\ \forall (x,t) \in D \times (0,T].
\end{equation}
$3)$ The existence and uniqueness of a solution to (\ref{mild formulation}) were established in \cite{BJW}. More specifically, the authors proved that, if $H \in \big{(} \frac{1}{2},1 \big{)}$, $d < 4H$ and deterministic initial data $u_{0}\in L^{p}(D)$ is given, then the problem (\ref{stoch. model}) admits a unique (global) mild solution $u\in C([0,T];L^{p}(D))$ for the exponents $p > 3$ specified in \cite[Proposition 4.1]{BJW}; by \cite[Remark 4.2]{BJW} every $p \in (5, 10]$ is admissible.
\end{remark} 
\section{Localization of $u(x,t)$ in $\mD^{1,2}$} \label{Section 3}
Consider the unique solution $u$ of (\ref{mild formulation}). The aim of this section is to prove that $u(x,t)\in\mD^{1,2}_{\textup{\text{loc}}}$ for all $(x,t)\in D\times[0,T]$. To do so, we are going to construct a suitable localizing sequence $\{(\Omega_{n},u_{n}(x,t))\}_{n \in \mN_{*}}$ as required by Definition \ref{Def. of local versions}. We will begin, however, by collecting a number of estimates for the Green's function and for the stochastic integral with which we shall work.
\subsection{Estimates for the Green's function and the stochastic integral} \label{Subsection 3.1}
$ $\newline \indent
Before we proceed with this subsection's main lemma, we need to mention some important estimates from \cite{WEB1}. More specifically, there exist constants $c,C>0$ such that, for any $x,y\in D$ and any $t\in(0,T]$, the following estimates hold:
\begin{equation} \label{Green's function's estimates}
\begin{gathered}
i)\ |G(x,y,t)| \leq c\, t^{-\frac{d}{4}} \exp\Big{(}-C\, |x-y|^{\frac{4}{3}}\, t^{-\frac{1}{3}} \Big{)},\quad
ii)\ |\nabla_{x} G(x,y,t)|\leq c\, t^{-\frac{d + 1}{4}} \exp\Big{(}-C\, |x-y|^{\frac{4}{3}}\, t^{-\frac{1}{3}}\Big{)}, \\
iii)\ |\Delta_y G(x,y,t)|\leq c\, t^{-\frac{d + 2}{4}} \exp\Big{(}-C\, |x-y|^{\frac{4}{3}}\, t^{-\frac{1}{3}}\Big{)}.
\end{gathered}
\end{equation}
Furthermore, given any constant $C>0$, there exists another constant $K = K(C)>0$ such that
\begin{equation} \label{specific calculation}
\int_{\mR^d} \exp \Big{(} -C\, |z|^{\frac{4}{3}}\, t^{-\frac{1}{3}} \Big{)}\, dz = K\, t^{\frac{d}{4}}.
\end{equation}
The following lemma collects, in the form in which they will be used, the consequences of (\ref{Green's function's estimates}), (\ref{specific calculation}) and of the expansion (\ref{Green's function}).
\begin{lemma} \label{Green's function's properties}
$1)$ There exists a constant $c = c(d) > 0$ such that
\begin{equation} \label{elem. estimate}
\sup_{x \in D} \int_D |\Delta_y G(x,y,t)|\, dy \leq c\, t^{-\frac{1}{2}},\ \forall t \in (0,T].
\end{equation}
Consequently, for every $v \in L^1 \big{(} [0,T]; L^{\infty}(D) \big{)}$ and every $0 \leq t_0 \leq t \leq T$, it holds that
\begin{equation} \label{consequence of elem estimate} 
\bigg{\|} \int_{t_0}^t \int_D |\Delta_y G(\cdot, y, t - s)|\, |v(y, s)|\, dyds \bigg{\|}_{L^{\infty}(D)} \leq c \int_{t_0}^t (t - s)^{-\frac{1}{2}} \| v(\cdot, s) \|_{L^{\infty}(D)}\, ds.
\end{equation}
$2)$ There exists a constant $c = c(d) > 0$ such that
\begin{equation} \label{space increment of G}
\big{\|} G(x, \diamond, t) - G(x', \diamond, t) \big{\|}_{L^{2}(D)} \leq c\, |x - x'|\, t^{-\frac{d + 2}{8}},\ \forall x', x \in D,\ \forall t \in (0,T].
\end{equation}
$3)$ For every $\rho \in [0,1]$ there exists a constant $c = c(d, \rho, T) > 0$ such that
\begin{equation} \label{time increment of G}
\big{\|} G(x, \diamond, t) - G(x, \diamond, t') \big{\|}^{2}_{L^{2}(D)} \leq c\, (t - t')^{2\rho}\, (t')^{-2\rho - \frac{d}{4}},\ \forall x \in D,\ \forall\, 0 < t' \leq t \leq T.
\end{equation}
\end{lemma}
\begin{proof}
$1)$ By (\ref{Green's function's estimates}) $iii)$ and (\ref{specific calculation}), for every $x \in D$ and every $t \in (0,T]$,
$$\int_D |\Delta_y G(x,y,t)|\, dy \leq c\, t^{-\frac{d + 2}{4}} \int_{\mR^d} \exp \Big{(} -C\, |x - y|^{\frac{4}{3}}\, t^{-\frac{1}{3}} \Big{)}\, dy = c\, K\, t^{-\frac{d + 2}{4} + \frac{d}{4}} = c\, K\, t^{-\frac{1}{2}},$$
which is (\ref{elem. estimate}). Bounding $|v(y,s)|$ by $\| v(\cdot,s) \|_{L^{\infty}(D)}$ and applying (\ref{elem. estimate}) at $t - s$ then gives (\ref{consequence of elem estimate}).

$2)$ Since $D$ is convex, the segment joining $x'$ to $x$ lies in $D$; setting $x_{\theta} := x' + \theta(x - x')$ for $\theta \in [0,1]$, the fundamental theorem of calculus applied to the map $\theta \mapsto G(x_{\theta}, y, t)$, followed by Minkowski's integral inequality in $L^{2}(D)$, gives
$$\big{\|} G(x, \diamond, t) - G(x', \diamond, t) \big{\|}_{L^{2}(D)} \leq |x - x'| \int_{0}^{1} \big{\|} \nabla_{x} G(x_{\theta}, \diamond, t) \big{\|}_{L^{2}(D)}\, d\theta,$$
while, by (\ref{Green's function's estimates}) $ii)$ and (\ref{specific calculation}),
$$\big{\|} \nabla_{x} G(x_{\theta}, \diamond, t) \big{\|}^{2}_{L^{2}(D)} \leq c\, t^{-\frac{d + 1}{2}} \int_{\mR^d} \exp \Big{(} -2C\, |x_{\theta} - y|^{\frac{4}{3}}\, t^{-\frac{1}{3}} \Big{)}\, dy = c\, K\, t^{-\frac{d + 1}{2} + \frac{d}{4}} = c\, K\, t^{-\frac{d + 2}{4}},$$
uniformly in $\theta \in [0,1]$; combining the two displays yields (\ref{space increment of G}).

$3)$ By (\ref{Green's function}) and the orthonormality of $\{a_{k}\}_{k \in \mN^d}$ in $L^{2}(D)$,
$$\big{\|} G(x, \diamond, t) - G(x, \diamond, t') \big{\|}^{2}_{L^{2}(D)} = \sum\limits_{k \in \mN^d} a_k^2(x) \big{(} e^{-|k|^4 t} - e^{-|k|^4 t'} \big{)}^2,$$
where the term indexed by $k = 0$ vanishes. Since $1 - e^{-z} \leq \min\{1, z\} \leq z^{\rho}$ for every $z \geq 0$,
$$0 \leq e^{-|k|^4 t'} - e^{-|k|^4 t} = e^{-|k|^4 t'} \big{(} 1 - e^{-|k|^4(t - t')} \big{)} \leq e^{-|k|^4 t'}\, |k|^{4\rho}\, (t - t')^\rho,$$
so that, using $|a_k| \leq (2/\pi)^{\frac{d}{2}}$ and then the fact $z^{2\rho} e^{-z} \leq c(\rho)\, e^{-\frac{z}{2}}$ (which holds since the function $z^{2\rho} e^{-\frac{z}{2}}$ is continuous and decays to 0 as $z \to +\infty$, meaning it has a global maximum $c(\rho)$) at $z = 2|k|^{4} t'$,
$$\big{\|} G(x, \diamond, t) - G(x, \diamond, t') \big{\|}^{2}_{L^{2}(D)} \leq c(d, \rho)\, (t - t')^{2\rho}\, (t')^{-2\rho} \sum\limits_{k \in \mN^d} e^{-|k|^4 t'}.$$
Finally $|k|^4 \geq \sum_{i = 1}^d k_i^4$, whence, for $t' \in (0,T]$,
$$\sum\limits_{k \in \mN^d} e^{-|k|^4 t'} \leq \bigg{(} \sum\limits_{j \in \mN} e^{-j^4 t'} \bigg{)}^{d} \leq \bigg{(} 1 + \int_{0}^{+\infty} e^{-u^4 t'}\, du \bigg{)}^{d} = \Big{(} 1 + \Gamma \big{(} \tfrac{5}{4} \big{)}\, (t')^{-\frac{1}{4}} \Big{)}^{d} \leq c(d,T)\, (t')^{-\frac{d}{4}}.$$
Combining the preceding estimates yields (\ref{time increment of G}).
\end{proof}
\begin{lemma} \label{Estimates for stoch. int.}
Let $H \in \big{(} \frac{1}{2},1 \big{)}$ and $d < 4H$. Then, the following hold. \\
$1)$ There exists a positive constant $c = c(d, H)$ such that, for every $0 \leq \zeta \leq t \leq T$,
\begin{equation} \label{sup estimate}
\sup_{x \in D} \int_{0}^{T} \int_{D} \big{|} \big{[} K^{*}_{H} \big{(} G(x, \diamond, t - \bullet) \mathbf{1}_{[\zeta, t]}(\bullet)  \big{)} \big{]} (y,s) \big{|}^2\, dy ds \leq c\, (t - \zeta)^{2H - \frac{d}{4}}.
\end{equation}
$2)$ The random field $\cn(x, t) := \int_0^t \int_D G(x, y, t - s)\, W_{H}(dy, ds)$, $(x,t) \in D \times [0,T]$, admits a modification which is jointly $\bar{\delta}$-H\"older continuous on $D \times [0,T]$ for every $\bar{\delta} \in \big{(} 0, H - \frac{d}{8} \big{)}$. We identify $\cn$ with this modification, which is unique up to indistinguishability, and it is with respect to it that the $L^{\infty}(D)$-norm below is computed. \\
$3)$ For any $p \geq 2$, there exists a positive constant $c = c(d, H, p, T)$ such that, for every $t \in [0,T]$,
\begin{equation} \label{stoch estimate}
E \Big{(} \big{\|} \cn(\cdot, t) \big{\|}_{L^{\infty}(D)}^{p} \Big{)} \leq c\, t^{\frac{p(8H - d - 2)}{8}}.
\end{equation}
\end{lemma}
\begin{proof}
For $0 \leq \zeta \leq t \leq T$ and $x \in D$ we define for brevity
$$\Phi(y, s\, ;\, \zeta, x, t) := G(x, y, t - s)\, \mathbf{1}_{[\zeta,t]}(s),\ (y,s) \in D \times [0,T].$$
By (\ref{Green's function's estimates}) $i)$, (\ref{specific calculation}) and 2) of Remark \ref{S.2}, we have $\Phi \in L^2(D \times [0,T]) \subseteq \ch$; thus the action of $K^{*}_{H}$ on $\Phi$ is well-defined and, by (\ref{fractional i.G.p}), $\cn(x,t) = \hat{W}_{H} \big{(} \Phi(\diamond, \bullet\, ;\, 0, x, t) \big{)}$. Being the image of a subset of $\ch$ under the isonormal Gaussian process $\hat{W}_{H}$, the family $\{\cn(x,t)\}_{(x,t) \in D \times [0,T]}$ is centered Gaussian, so that for every $q \geq 2$ and every $\phi, \psi \in \ch$
\begin{equation} \label{Gaussian moments}
E \big{(} |\hat{W}_{H}(\phi) - \hat{W}_{H}(\psi)|^{q} \big{)} = c(q)\, \| \phi - \psi \|^{q}_{\ch}.
\end{equation}
All three assertions rest on the following elementary identity: for every $\alpha \in [0, H)$ and every $0 \leq \zeta \leq t \leq T$,
\begin{equation} \label{beta identity}
\int_{\zeta}^{t} \int_{\zeta}^{t} (t - s)^{-\alpha}\, (t - r)^{-\alpha}\, |s - r|^{2H - 2}\, ds dr = c(H,\alpha)\, (t - \zeta)^{2H - 2\alpha}.
\end{equation}
Indeed, by symmetry the left-hand side equals twice its restriction to $\{r < s\}$, on which the substitution $s = r + (t - r)v$ gives
$$\int_{r}^{t} (s - r)^{2H - 2} (t - s)^{-\alpha}\, ds = \text{B} \big{(} 2H - 1, 1 - \alpha \big{)}\, (t - r)^{2H - 1 - \alpha}$$
(both arguments of the beta function are positive, since $H > \frac{1}{2}$ and $\alpha < H < 1$); multiplying by $(t - r)^{-\alpha}$ and integrating in $r$ yields (\ref{beta identity}) with $c(H,\alpha) = 2 \text{B} \big{(} 2H - 1, 1 - \alpha \big{)} / (2H - 2\alpha)$.

1) Since $K_{H}^{*}$ is an isometry between $\ch$ and $L^{2}(D \times [0,T])$, we have
\begin{align*}
\int_{0}^{T} \int_{D} \big{|} \big{[} K^{*}_{H} \big{(} G(x, \diamond, t - \bullet) \mathbf{1}_{[\zeta,t]}(\bullet)  \big{)} \big{]} (y, s) &\big{|}^{2}\, dy ds = \big{\|} \big{[} K_H^*(\Phi) \big{]} (\diamond, \bullet\, ;\, \zeta, x, t) \big{\|}^{2}_{L^2(D \times [0,T])} = \big{\|} \Phi (\diamond, \bullet\, ;\, \zeta, x, t) \big{\|}^{2}_{\ch} \\
\overset{(\ref{S.3})}&{=} c(H) \int_D \int_\zeta^t \int_\zeta^t
G(x, y, t - s)\, G(x, y, t - r)\, |s - r|^{2H - 2}\, ds dr dy.
\end{align*}
We bound the integrand by its absolute value and integrate in $y$ first by Tonelli's theorem; the Cauchy-Schwarz inequality in $y$, together with (\ref{int. property of Green's function}) applied with $y$ replaced by $x$ and $t_{1} = t_{2} = t - s$ (respectively $t_{1} = t_{2} = t - r$), then gives
$$\int_{D} |G(x, y, t - s)|\, |G(x, y, t - r)|\, dy \leq G \big{(} x, x, 2(t - s) \big{)}^{\frac{1}{2}}\, G \big{(} x, x, 2(t - r) \big{)}^{\frac{1}{2}} \leq c(d)\, (t - s)^{-\frac{d}{8}} (t - r)^{-\frac{d}{8}},$$
where for the last inequality we used that $\tau \mapsto G(x, x, \tau)$ is non-increasing, by (\ref{Green's function}), so that $G \big{(} x, x, 2(t-s) \big{)} \leq G \big{(} x, x, t - s \big{)}$, and the upper bound of (\ref{diagonal bounds}) applies at $t - s$; similarly for $r$. Applying now (\ref{beta identity}) with $\alpha = \frac{d}{8}$, which is admissible since $d < 4H$ gives $\frac{d}{8} < \frac{H}{2} < H$, we obtain
\begin{equation} \label{H norm of Phi}
\big{\|} \Phi(\diamond, \bullet\, ;\, \zeta, x, t) \big{\|}^{2}_{\ch} \leq c(d, H)\, (t - \zeta)^{2H - \frac{d}{4}}.
\end{equation}
As this bound does not depend on $x$, (\ref{sup estimate}) follows.

2) Fix $\rho \in \big{(} 0, H - \frac{d}{8} \big{)}$, a non-empty range since $d < 4H$, and note that $\rho < H < 1$. We estimate separately, for $x, x' \in D$ and $0 \leq t' \leq t \leq T$, the increments of $\Phi(\diamond, \bullet\, ;\, 0, x, t)$ in each variable.

The difference $\Phi(\diamond, \bullet\, ;\, 0, x, t) - \Phi(\diamond, \bullet\, ;\, 0, x', t)$ lies in $L^{2}(D \times [0,T]) \subseteq \ch$, so that (\ref{S.3}) applies to it; bounding the integrand by its absolute value, integrating in $y$ first by Tonelli's theorem and using the Cauchy-Schwarz inequality in $y$ together with (\ref{space increment of G}) at $t - s$ and at $t - r$, we get, by (\ref{beta identity}) with $\alpha = \frac{d + 2}{8}$,
\begin{equation} \label{space increment in H}
\big{\|} \Phi(\diamond, \bullet\, ;\, 0, x, t) - \Phi(\diamond, \bullet\, ;\, 0, x', t) \big{\|}^{2}_{\ch} \leq c(d, H)\, |x - x'|^{2}\, t^{2H - \frac{d + 2}{4}} \leq c(d, H, T)\, |x - x'|^{2\rho}
\end{equation}
for all $x, x' \in D$ and all $t \in [0,T]$; here $\alpha = \frac{d+2}{8} < H$ since $d + 2 < 4H + 2 < 8H$ by $d < 4H$ and $H > \frac{1}{2}$, while for the last inequality we used $2\rho \leq 2$, the boundedness of $D$ and 
$t \leq T$.

We turn to the time increment. Up to a Lebesgue-null set of times,
$$\Phi(\diamond, \bullet\, ;\, 0, x, t) - \Phi(\diamond, \bullet\, ;\, 0, x, t') = \psi_{1} + \psi_{2},\quad \psi_{2} := \Phi(\diamond, \bullet\, ;\, t', x, t),$$
$$\psi_{1}(y,s) := \big{(} G(x, y, t - s) - G(x, y, t' - s) \big{)} \mathbf{1}_{[0,t']}(s),\ (y,s) \in D \times [0,T].$$
The second term is controlled by (\ref{H norm of Phi}) with $\zeta = t'$, which gives $\| \psi_{2} \|^{2}_{\ch} \leq c(d,H)\, (t - t')^{2H - \frac{d}{4}} \leq c(d,H,T)\, (t - t')^{2\rho}$, since $2\rho < 2H - \frac{d}{4}$ and $t - t' \leq T$. For the first, $\psi_{1} \in L^{2}(D \times [0,T]) \subseteq \ch$, so that, arguing as above and using (\ref{time increment of G}) with $t$ and $t'$ replaced by $t - s$ and $t' - s$ respectively,
$$\| \psi_{1} \|^{2}_{\ch} \leq c(H) \int_{0}^{t'} \int_{0}^{t'} \big{\|} \psi_{1}(\diamond, s) \big{\|}_{L^{2}(D)} \big{\|} \psi_{1}(\diamond, r) \big{\|}_{L^{2}(D)} |s - r|^{2H - 2} ds dr$$
$$\leq c(d, H, \rho, T)\, (t - t')^{2\rho} \int_{0}^{t'} \int_{0}^{t'} (t' - s)^{-\rho - \frac{d}{8}} (t' - r)^{-\rho - \frac{d}{8}} |s - r|^{2H - 2}\, ds dr = c\, (t - t')^{2\rho}\, (t')^{2H - 2\rho - \frac{d}{4}},$$
by (\ref{beta identity}) with $\alpha = \rho + \frac{d}{8}$, admissible precisely because $\rho < H - \frac{d}{8}$; the same inequality also makes the exponent $2H - 2\rho - \frac{d}{4}$ positive, so that this last factor is bounded by $T^{2H - 2\rho - \frac{d}{4}}$. Consequently
\begin{equation} \label{time increment in H}
\big{\|} \Phi(\diamond, \bullet\, ;\, 0, x, t) - \Phi(\diamond, \bullet\, ;\, 0, x, t') \big{\|}_{\ch} \leq c(d, H, \rho, T)\, |t - t'|^{\rho},\ \forall x \in D,\ \forall t, t' \in [0,T].
\end{equation}
Combining (\ref{space increment in H}) and (\ref{time increment in H}) by the triangle inequality in $\ch$, and using that $a^{\rho} + b^{\rho} \leq 2 (a^{2} + b^{2})^{\frac{\rho}{2}}$ for all $a, b \geq 0$, we conclude from (\ref{Gaussian moments}) that, for every $q \geq 2$,
$$E \big{(} |\cn(x, t) - \cn(x', t')|^{q} \big{)} \leq c(d, H, q, \rho, T) \big{(} |x - x'|^{2} + |t - t'|^{2} \big{)}^{\frac{q \rho}{2}},\ \forall (x,t), (x',t') \in D \times [0,T].$$
Choosing $q$ such that $q \rho > d + 1$, the $(d+1)$-dimensional Kolmogorov-Chentsov theorem (see \cite[Section 1.4]{Kun}) yields a modification of $\cn$ which is jointly H\"older continuous on $D \times [0,T]$ of any order smaller than $\rho - \frac{d + 1}{q}$. Such a modification is unique up to indistinguishability, so letting $q \to +\infty$ and $\rho \uparrow H - \frac{d}{8}$ we obtain a single modification which is jointly $\bar{\delta}$-H\"older continuous for every $\bar{\delta} \in \big{(} 0, H - \frac{d}{8} \big{)}$, as claimed. We identify $\cn$ with it from now on.

3) We fix throughout an arbitrary exponent $p > d$; the remaining range $p \in [2,d]$, non-empty only when $d \in \{2,3\}$, then follows from the case $p = 4$ by H\"older's inequality for the expectation. By 2), $x \mapsto \cn(x,t)$ is continuous for every $t \in [0,T]$, so that $\sup_{x \in D} |\cn(x,t)|$ is a well-defined random variable and coincides with $\| \cn(\cdot, t) \|_{L^{\infty}(D)}$. The $d$-dimensional Garsia-Rodemich-Rumsey lemma (see \cite{GRR} for $d = 1$ and \cite{G} for $d \geq 2$), applied with $\varPsi(u) = u^{p}$ and $\varrho(u) = u^{\alpha}$, where $\alpha \in \big{(} \frac{2d}{p}, 1 + \frac{d}{p} \big{)}$ (a non-empty range precisely because $p > d$), gives
$$\| \cn(\cdot, t) \|_{L^{\infty}(D)} \leq |\cn(x_{0}, t)| + c(d,p)\, \bigg{(} \int_{D} \int_{D} \frac{|\cn(x, t) - \cn(y, t)|^{p}}{|x - y|^{\alpha p}}\, dxdy \bigg{)}^{\frac{1}{p}}$$
for any fixed $x_{0} \in D$ (the finite constant $c(d,p)$ is a direct consequence of the lower bound $\alpha > \frac{2d}{p}$). On the other hand, (\ref{Gaussian moments}) combined with (\ref{H norm of Phi}) at $\zeta = 0$, then with $t \leq T$, and with the first inequality of (\ref{space increment in H}), gives
$$E \big{(} |\cn(x, t)|^{p} \big{)} \leq c(d, H, p)\, t^{\frac{p(8H - d)}{8}} \leq c(d, H, p, T)\, t^{\frac{p(8H - d - 2)}{8}},$$
$$E \big{(} |\cn(x, t) - \cn(x', t)|^{p} \big{)} \leq c(d, H, p)\, |x - x'|^{p}\, t^{\frac{p(8H - d - 2)}{8}},$$
uniformly in $x, x' \in D$. Taking $p$-th power and expectation in the Garsia-Rodemich-Rumsey bound above and applying Tonelli's theorem, we conclude that
$$E \Big{(} \big{\|} \cn(\cdot, t) \big{\|}^{p}_{L^{\infty}(D)} \Big{)} \leq t^{\frac{p(8H - d - 2)}{8}} \bigg{(} c(d, H, p, T) + c(d, H, p) \int_{D} \int_{D} |x - y|^{p(1 - \alpha)}\, dxdy \bigg{)} = c(d, H, p, T)\, t^{\frac{p(8H - d - 2)}{8}}$$
(the double integral is finite if and only if $p(1 - \alpha) > -d \iff \alpha < 1 + \frac{d}{p}$), which is (\ref{stoch estimate}).
\end{proof}
\subsection{Regularity of the solution $u$} \label{Subsection 3.2}
$ $\newline \indent
Regarding the regularity of $u$, each term on the right-hand side of (\ref{mild formulation}) has been separately studied and various results have been obtained depending on the regularity of the initial data $u_{0}$. Our hypotheses and initial data are in accordance with the assumptions 1-4 of \cite{WEB1} and Theorem 1.1 of \cite{BJW}. Thus, we can invoke results therein to prove the following lemma:
\begin{lemma} \label{continuity of u}
The solution $u$ of $(\ref{mild formulation})$ admits a modification which is jointly continuous on $D \times [0,T]$.
\end{lemma}
\begin{proof}
The right-hand side of (\ref{mild formulation}) consists of three terms:
$$A_{1}(x,t) := \int_{D} G(x, y, t) u_{0}(y)\, dy,\quad A_{2}(x, t) := \int_{0}^{t} \int_{D} \Delta_y G(x, y, t - s) f \big{(} u(y, s) \big{)}\, dy ds,$$
$$A_{3}(x,t) := \sigma \int_{0}^{t} \int_{D} G(x,y,t - s) W_{H} (dy,ds).$$

Since $u_{0}$ is deterministic and continuous, Lemma 2.1 of \cite{WEB1} guarantees that $A_{1}$ is jointly continuous on $D \times [0,T]$, while the term $A_{3}$ is jointly H\"older continuous on $D \times [0,T]$ by 2) of Lemma \ref{Estimates for stoch. int.}.

Lastly, pages 797-799 of \cite{WEB1} are dedicated to the study of $A_{2}$. Fix any $q \in \big{(} 3 \vee \frac{3d}{2}, p \big{]}$, where $p$ is the exponent of Theorem 1.1 of \cite{BJW}; such a $q$ exists, since by Remark 4.2 of \cite{BJW} every $p \in (5, 10]$ is admissible. Since $u\in C([0,T];L^{p}(D))$ and $D$ is bounded, it also holds that $u \in L^{\infty}([0,T];L^{q}(D))$. It can be shown using a factorization method that 
$$A_2(x,t) = \frac{\sin(\pi \alpha)}{\pi}\, \mathcal{F} \big{(} \mathcal{K} ( u ) \big{)} (x,t),\ \forall (x,t) \in D \times [0,T],$$ 
where $\alpha \in \big{(} \frac{1}{2} + \frac{d}{2q}, 1 - \frac{d}{4q} \big{)}$ is fixed and
\begin{align*}
\mathcal{F}(v)(x,t)&:=\int_{0}^{t}\int_{D} G(x,z,t-s)(t-s)^{-\alpha}v(z,s)\, dzds,\\
\mathcal{K}(v)(z,s) &:= \int_{0}^{s} \int_{D} \Delta_y G(z, y, s - s') (s - s')^{\alpha-1}f\big{(}v(y,s')\big{)}\, dyds'.
\end{align*}
The author then shows that the operator $\mathcal{K}$ maps $L^{\infty}([0,T];L^{q}(D))$ into itself and $\mathcal{F}(v)$ is jointly H\"older continuous whenever $v \in L^{\infty}([0,T];L^{q}(D))$. 

Combining all the above shows that $A_1+A_2+A_3$ admits a modification jointly continuous on $D \times [0,T]$ and so does $u$, as claimed.
\end{proof}
\noindent We henceforth identify $u$ with its jointly continuous modification and retain the same notation for it; the stochastic term has already been so identified in Lemma \ref{Estimates for stoch. int.}. These identifications are used implicitly throughout the remainder of the paper, for example, in the definition and well-posedness of the sets $\Omega_n$ in (\ref{def. of Omega_n}) and in every regularity argument of Lemma \ref{main lemma 1}.
\subsection{Existence of a localizing sequence $\{(\Omega_{n},u_{n}(x,t))\}_{n \in \mN_{*}}$ (Part 1)} \label{Subsection 3.3}
$ $\newline \indent
This subsection's purpose is to define a sequence $\{\Omega_{n}\}_{n\in\mN_{*}} \subset \cf$ such that $\Omega_{n}\uparrow\Omega$ a.s. and construct a sequence of processes $\{u_{n}\}_{n\in\mN_{*}}$  such that $u_{n}(x,t)=u(x,t)$ for all $(x,t)\in D\times[0,T]$ a.s. on $\Omega_{n}$. 
\\[0.3cm]
\indent We define the sequence of sets $\{\Omega_{n}\}_{n\in\mN_{*}}$ as
\begin{equation} \label{def. of Omega_n}
    \Omega_{n}:= \bigg{\{} \omega \in \Omega:\; \sup_{t \in [0,T]}\sup_{x \in D} |u(x,t;\omega)| < n \bigg{\}},\ n\in\mN_{*}.
\end{equation}

\begin{remark} 
If $A:=\{\omega\in\Omega\, :\, (x,t) \mapsto u(x,t;\omega)\ \text{is continuous}\ \text{on}\ D \times [0,T]\}$, then it is easy to verify that $\displaystyle A \subseteq \bigcup\limits_{n=1}^{+\infty} \Omega_{n}$. Since $\mP(A)=1$, it holds that $\displaystyle \mP\bigg{(}\bigcup\limits_{n=1}^{+\infty}\Omega_{n}\bigg{)}=1$. Moreover, notice that $\Omega_1\subseteq\Omega_2\subseteq\cdots\subseteq \Omega$. Thus, due to the monotonicity of $\{\Omega_{n}\}_{n\in\mN_{*}}$, we obtain
$\displaystyle \lim\limits_{n\to+\infty} \mP(\Omega_{n}) = \mP \bigg{(} \bigcup\limits_{n=1}^{+\infty}\Omega_{n}\bigg{)}=1$. According to Definition \ref{Def. of local versions}, we have shown that $\Omega_{n}\uparrow\Omega$ almost surely.
\end{remark}
Before we proceed, we introduce some preliminary constructions following the approach of Cardon-Weber \cite{WEB1}. Let $\{\chi_n\}_{n \in \mN_*}$ be a sequence of $C^{1}$ cut-off functions $\chi_{n}: [0,+\infty) \rightarrow [0,+\infty)$ satisfying the conditions
$$|\chi_{n}(x)| \leq 1,\ |\chi'_{n}(x)| \leq 2,\ \forall x \in [0, +\infty)\quad \&\quad \chi_{n}(x):=
\begin{cases}
1, & \text{if  $x \in [0, n]$},\\
0, & \text{if $x \in [n+1, +\infty)$}.
\end{cases}$$
and consider the sequence of functions $f_{n} : \mR \to \mR$ with $f_{n}(x) := \chi_{n}(|x|) f(x)$. It is not hard to prove that, for fixed $n\in\mN_{*}$, $f_{n}\in C^{1}(\mR)$ and there exists a positive constant $c = c(f, n)$ such that $|f_n(x)| \leq c$ and $|f_n'(x)| \leq c$ for every $x \in \mR$. As a consequence of the second bound, the function $f_{n}$ is Lipschitz, with Lipschitz constant $c$.

Let now $n\in\mN_{*}$ be fixed. We consider the following cut-off stochastic integral equation:
\begin{align} \label{mild formulation of u_n}
\begin{split}
u_{n}(x, t) = \int_{D} G(x, y, t) u_{0}(y)\, dy &+ \int_{0}^{t} \int_{D} \Delta_y G(x, y, t - s) f_{n} \big{(} u_{n}(y, s) \big{)}\, dy ds \\
&+ \sigma \int_{0}^{t} \int_{D} G(x, y, t - s) W_{H}(dy, ds),\ (x, t) \in D \times [0,T],
\end{split}
\end{align}
and we will prove the following lemma:
\begin{lemma} \label{main lemma 1}
A solution $u_{n}$ to equation $(\ref{mild formulation of u_n})$ exists and admits a modification which is jointly continuous on $D \times [0, T]$. Furthermore, uniqueness of the solution holds, up to indistinguishability, within the class of processes admitting a jointly continuous version. Lastly, for any $p\geq2$, there exists a positive constant $\cc = \cc(\sigma, d, f, H, n, p, T, u_{0})$ such that
\begin{equation} \label{sup estimate of u_n}
\sup\limits_{t \in [0,T]} E \Big{(} \| u_{n}(\cdot, t) \|_{L^{\infty}(D)}^{p} \Big{)} \leq \cc.
\end{equation}
\end{lemma}
\begin{proof}
We fix, throughout this proof, an arbitrary exponent $p \geq 2$. For every $(x, t) \in D \times[0,T]$, we define
$$u_{n,0}(x,t):=\int_{D} G(x,y,t) u_{0}(y)\, dy$$
and, for any $k \in \mN$, we consider the following Picard iteration scheme:
\begin{align} \label{Picard scheme}
\begin{split}
u_{n,k+1}(x, t) = u_{n,0}(x, t) &+ \int_{0}^{t} \int_{D} \Delta_y G(x, y, t - s) f_{n} \big{(} u_{n,k}(y, s) \big{)}\, dy ds \\
&+ \sigma \int_{0}^{t} \int_{D} G(x, y, t - s) W_{H}(dy,ds).
\end{split}
\end{align}

The iterates $\{u_{n,k}\}_{k \in \mN}$ are $\cf_{t}$-adapted and jointly continuous on $D \times [0,T]$; this is proved by induction on $k$. The base case $u_{n,0}$ is deterministic, hence adapted, and jointly continuous. Assume now that $u_{n,k}$ has both properties. The second summand of (\ref{Picard scheme}) is then jointly H\"older continuous by the factorization argument of Lemma \ref{continuity of u}, applied with the same pair $(\alpha,q)$. Here that argument is in fact simpler: since $f_{n}$ is bounded, $f_{n}(u_{n,k})$ belongs to $L^{\infty} \big{(} [0,T];L^{q}(D) \big{)}$ for every $k$, so that no a priori estimate on the iterate is required. This summand is also adapted, as it involves only the values of $u_{n,k}$ up to time $t$. Finally, the third summand is $\cf_{t}$-measurable and has already been identified with its jointly H\"older continuous modification. Accordingly, all norms $\|\cdot\|_{L^{\infty}(D)}$ and spaces $L^{p} \big{(} \Omega;L^{\infty}(D) \big{)}$ below are understood with respect to the separable Banach space $\big{(} C(D),\|\cdot\|_{\infty} \big{)}$, on which the two norms agree; this resolves the measurability issues arising from the non-separability of $L^{\infty}(D)$.

We now follow the procedure in the proof of \cite[Lemma 3.2]{WEB1} in order to derive a recursive estimate for the differences of successive iterates. The stochastic term of (\ref{Picard scheme}) does not depend on $k$, hence it cancels in the difference $u_{n,k+1} - u_{n,k}$. Using the Lipschitz property of $f_{n}$, taking the supremum over $x \in D$ and applying (\ref{consequence of elem estimate}), we obtain 
$$\| u_{n,k+1}(\cdot,t) - u_{n,k}(\cdot,t) \|_{L^{\infty}(D)} \leq c(d, f, n) \int_{0}^{t} (t - s)^{-\frac{1}{2}} \| u_{n,k}(\cdot, s) - u_{n,k-1}(\cdot, s) \|_{L^{\infty}(D)}\, ds.$$
We then raise to the $p$-th power, split the kernel as $(t-s)^{-\frac{1}{2}} = (t-s)^{-\frac{p-1}{2p}} \cdot (t-s)^{-\frac{1}{2p}}$ and apply H\"older's inequality with the conjugate exponents $\frac{p}{p-1}$ and $p$ to the two factors respectively. Since $\int_{0}^{t} (t-s)^{-\frac{1}{2}} ds = 2\sqrt{t} \leq 2\sqrt{T}$, taking expectation and using Tonelli's theorem yields
\begin{equation} \label{recursive}
E \Big{(} \| u_{n,k+1}(\cdot, t) - u_{n,k}(\cdot, t) \|_{L^{\infty}(D)}^{p} \Big{)}
\leq c(d, f, n, p, T) \int_{0}^{t} (t-s)^{-\frac{1}{2}}\, E \Big{(} \| u_{n,k}(\cdot, s) - u_{n,k-1}(\cdot, s) \|_{L^{\infty}(D)}^{p} \Big{)} ds
\end{equation}
for every $k \in \mN_{*}$ and every $t \in [0,T]$.

Before proceeding, we note that there exists a positive constant $\cc_{1} = \cc_{1}(\sigma, d, f, H, n, p, T)$ such that
\begin{equation} \label{sup bound of dif. of first 2 terms}
\sup\limits_{t\in[0,T]}E \Big{(} \| u_{n,1}(\cdot,t)-u_{n,0}(\cdot,t) \|_{L^{\infty}(D)}^{p} \Big{)}\leq \cc_{1}.
\end{equation}
Indeed, by (\ref{Picard scheme}) for $k = 0$, the boundedness of $f_{n}$ together with (\ref{elem. estimate}) bounds the drift term by $c(d, f, n, p, T)$ uniformly in $t \in [0,T]$, while (\ref{stoch estimate}) bounds the stochastic term by $c(\sigma, d, H, p, T)$; their sum is the desired $\cc_{1}$.

We temporarily denote by $\tilde{c}$ the constant of (\ref{recursive}) and claim that
\begin{equation} \label{induction for diff}
E \Big{(} \| u_{n,k}(\cdot,t) - u_{n,k-1}(\cdot,t) \|_{L^{\infty}(D)}^{p} \Big{)} \leq \cc_{1}\, \frac{\big{(} \tilde{c}\, \sqrt{\pi} \big{)}^{k-1}\, t^{\frac{k-1}{2}}}{\Gamma \big{(} \frac{k+1}{2} \big{)}},\ \forall t \in [0,T],\ \forall k \in \mN_{*}.
\end{equation}
Indeed, the case $k=1$ follows from (\ref{sup bound of dif. of first 2 terms}). Assume now that (\ref{induction for diff}) holds for some $k \in \mN_{*}$. Recalling that
$$\int_{0}^{t} (t - s)^{-\frac{1}{2}} s^{\frac{k-1}{2}}\, ds = \text{B} \Big{(} \tfrac{1}{2}, \tfrac{k + 1}{2} \Big{)}\, t^{\frac{k}{2}} = \frac{\sqrt{\pi}\ \Gamma \big{(} \frac{k + 1}{2} \big{)}}{\Gamma \big{(} \frac{k}{2} + 1 \big{)}}\, t^{\frac{k}{2}},$$
relation (\ref{recursive}) gives
$$E \Big{(} \| u_{n,k + 1}(\cdot,t) - u_{n,k}(\cdot,t) \|_{L^{\infty}(D)}^{p} \Big{)} \leq \tilde{c}\, \cc_{1}\, \frac{\big{(} \tilde{c}\, \sqrt{\pi} \big{)}^{k-1}}{\Gamma \big{(} \frac{k+1}{2} \big{)}} \int_{0}^{t} (t-s)^{-\frac{1}{2}} s^{\frac{k-1}{2}}\, ds = \cc_{1}\, \frac{\big{(} \tilde{c}\, \sqrt{\pi} \big{)}^{k}\, t^{\frac{k}{2}}}{\Gamma \big{(} \frac{k}{2} + 1 \big{)}},$$
which is (\ref{induction for diff}) for $k+1$, thus closing the induction. Raising (\ref{induction for diff}) to the $\frac{1}{p}$-power and taking the supremum over $t \in [0,T]$, we obtain
\begin{equation} \label{sup bound of dif. of sequential terms}
\| u_{n,k} - u_{n,k - 1} \|_{L^{\infty}([0,T];L^{p}(\Omega;L^{\infty}(D)))} \leq \bigg{(} \cc_1\, \frac{\big{(} c(d, f, n, p, T)\, \sqrt{\pi} \big{)}^{k - 1}\, T^{\frac{k - 1}{2}}}{\Gamma \big{(} \frac{k + 1}{2} \big{)}} \bigg{)}^{\frac{1}{p}} =: \beta_k,\ \forall k \in \mN_{*}.
\end{equation}
Since $\Gamma \big{(} \frac{k+1}{2} \big{)} / \Gamma \big{(} \frac{k}{2} + 1 \big{)} \to 0$ as $k \to +\infty$, we have $\beta_{k+1} / \beta_{k} \to 0$, so the ratio test gives
\begin{equation} \label{constant C_2}
\cc_{2} := \sum\limits_{k = 1}^{+\infty} \beta_k < +\infty\ (\text{dependent on}\ \sigma, d, f, H, n, p, T).
\end{equation}
As a consequence of (\ref{sup bound of dif. of sequential terms}) and (\ref{constant C_2}), $\{u_{n,k}\}_{k\in\mN}$ is Cauchy in  $L^{\infty}\big{(}[0,T];L^{p}(\Omega;L^{\infty}(D))\big{)}$. Since the supremum in (\ref{sup bound of dif. of sequential terms}) already runs over all of $[0,T]$, $\{u_{n,k}(\cdot,t)\}_{k \in \mN}$ is Cauchy in $L^p(\Omega;L^\infty(D))$ at a rate uniform in $t$. By completeness, this gives, for every $t \in [0,T]$, a limit $u_n(\cdot,t) \in L^p(\Omega;L^\infty(D))$ with
\begin{equation} \label{first limit}
\sup_{t \in [0,T]} \| u_{n,k}(\cdot,t) - u_{n}(\cdot,t) \|_{L^{p}(\Omega;L^{\infty}(D))} \to 0\ \text{as}\ k \to +\infty
\end{equation}
and in particular
\begin{equation} \label{second limit}
\sup_{(x,t) \in D \times [0,T]} \| u_{n,k}(x,t) - u_{n}(x,t) \|_{L^{p}(\Omega)} \to 0\ \text{as}\ k \to +\infty.
\end{equation}
Since $L^{p}(\Omega; C(D))$ is closed in $L^{p}(\Omega; L^{\infty}(D))$ and the iterates are $C(D)$-valued, so is $u_{n}(\cdot,t)$ almost surely; thus $u_{n}(x,t)$ is a well-defined random variable for every $(x,t)$, and it is $\cf_{t}$-adapted by (\ref{second limit}) and the completeness of the filtration.

We next verify that $u_n$ satisfies (\ref{mild formulation of u_n}) almost surely, for every $(x,t) \in D \times [0,T]$. By Tonelli's theorem, $\{u_{n,k}\}_{k\in\mN}$ is Cauchy in $L^p(\Omega \times D \times [0,T])$ as well, and (\ref{second limit}) identifies its limit there with $u_n$. Thus $u_{n}$ is jointly measurable and the nonlinear term of (\ref{mild formulation of u_n}) is well-defined. Subtracting (\ref{Picard scheme}) from (\ref{mild formulation of u_n}), taking $L^p(\Omega)$-norms and using Minkowski's integral inequality together with the Lipschitz property of $f_n$ and (\ref{elem. estimate}), we bound $\| u_n(x,t) - \text{r.h.s. of } (\ref{mild formulation of u_n}) \|_{L^p(\Omega)}$ by
$$ \| u_n(x, t) - u_{n, k + 1}(x,t) \|_{L^p(\Omega)} + c(d, f, n, T) \sup_{(y, s) \in D \times [0,T]} \| u_{n,k}(y, s) - u_n(y, s) \|_{L^p(\Omega)},$$
both terms vanishing as $k \to +\infty$ by (\ref{second limit}). The argument employed previously for the iterates then shows that the right-hand side of (\ref{mild formulation of u_n}) is jointly continuous, so that $u_n$ admits a jointly continuous modification, again denoted $u_n$. Both sides of (\ref{mild formulation of u_n}) being jointly continuous, a standard argument on a countable dense subset of $D \times [0,T]$ produces a single null set off which equation (\ref{mild formulation of u_n}) holds for all $(x,t) \in D \times [0,T]$ simultaneously.

We now turn to the uniqueness of the solution. Suppose there exists another jointly continuous solution $w_{n}$ of (\ref{mild formulation of u_n}) and let $N$ denote the union of the two null sets, outside of which both processes satisfy (\ref{mild formulation of u_n}) for every $(x,t) \in D \times [0,T]$ simultaneously. Subtracting their respective equations, taking absolute values on both sides and using the Lipschitz property of $f_{n}$, then taking the supremum over $x \in D$ and using (\ref{consequence of elem estimate}), we get, for every $\omega \in N^c := \Omega - N$,
$$\delta_n(t;\omega) \leq c(d, f, n) \int_{0}^{t} (t - s)^{-\frac{1}{2}} \delta_{n}(s;\omega)\, ds,\ \forall t \in [0,T],$$
where $\delta_n(t;\omega) := \| u_n(\cdot,t;\omega) - w_n(\cdot,t;\omega) \|_{L^{\infty}(D)},\, t \in [0,T].$ We view this as a deterministic Volterra inequality and note that $\delta_n(t;\omega)$ is non-negative and continuous on $[0,T]$ for every $\omega \in N^c$. Thus, by the singular Gr\"onwall inequality (see \cite[Lemma 7.1.1]{He}), we derive $\delta_n(t;\omega) = 0$ for every $t \in [0,T]$ and every $\omega \in N^c$. Hence
$$\mP ( \{ \omega \in \Omega\, :\, u_n(x,t;\omega) = w_n(x,t;\omega),\ \forall (x,t) \in D \times [0,T]\}) \geq \mP(N^c) = 1,$$
which proves the indistinguishability of the two processes.

Lastly, we prove (\ref{sup estimate of u_n}). For every $t \in [0,T]$, the triangle inequality, (\ref{sup bound of dif. of sequential terms}) and (\ref{constant C_2}) give
\begin{align} \label{Random bound}
\begin{split}
\| u_{n,k}(\cdot,t) \|_{L^{p}(\Omega;L^{\infty}(D))} &\leq \| u_{n,0}(\cdot,t) \|_{L^{p}(\Omega;L^{\infty}(D))} + \sum\limits_{i = 1}^{k} \| u_{n,i} - u_{n,i-1} \|_{L^{\infty}([0,T];L^{p}(\Omega;L^{\infty}(D)))} \\
&\leq c\, \| u_{0} \|_{L^{\infty}(D)} + \cc_2,\ \forall k \in \mN,
\end{split}
\end{align}
where we also used (\ref{Green's function's estimates}) $i)$ and (\ref{specific calculation}) for the first term. Letting $k \to +\infty$ by means of (\ref{first limit}), we obtain (\ref{sup estimate of u_n}) with $\cc := \big{(} c\, \| u_{0} \|_{L^{\infty}(D)} + \cc_2 \big{)}^p$, which concludes the proof.
\end{proof}

The only thing left to do in this subsection is to prove that $u_{n}(x, t) = u(x, t)$ a.s. on $\Omega_{n}$ for every $(x,t)\in D\times[0,T]$. To do so, we start by subtracting (\ref{mild formulation of u_n}) and (\ref{mild formulation}) to get
\begin{align*}
u_{n}(x,t) - u(x,t) &= \int_{0}^{t} \int_{D} \Delta_y G(x, y, t - s) \left ( f_n \big{(} u_{n}(y,s) \big{)} - f \big{(} u(y,s) \big{)} \right ) dy ds.
\end{align*}
Using now the fact that $f(u) = f_{n}(u)$ on $\Omega_n$, we can follow a procedure similar to the one in the previous lemma to obtain, for any $p \geq 2$,
$$E \Big{(} \mathbf{1}_{\Omega_n} \| u_{n}(\cdot,t) - u(\cdot,t) \|_{L^{\infty}(D)}^{p} \Big{)} \leq c(d, f, n, p, T) \int_{0}^{t} (t-s)^{-\frac{1}{2}}\, E \Big{(} \mathbf{1}_{\Omega_n} \| u_{n}(\cdot,s)-u(\cdot,s) \|_{L^{\infty}(D)}^{p} \Big{)}\, ds,$$
the left-hand side being bounded uniformly in 
$t$, by (\ref{sup estimate of u_n}) and the definition (\ref{def. of Omega_n}) of $\Omega_{n}$.
Thus, by the singular Gr\"onwall inequality, we get
$\| u_{n}(x,t) - u(x,t) \|_{L^p(\Omega_n)} = 0$ for all $(x,t) \in D \times [0,T]$.
This yields that $u_n(x,t) = u(x,t)$ almost surely on $\Omega_{n}$ for any $(x,t) \in D \times [0,T]$ which, combined with their almost sure space-time continuity (Lemma \ref{continuity of u} for $u$ and Lemma \ref{main lemma 1} for $u_n$), implies their indistinguishability on $\Omega_{n}$.
\subsection{Existence of a localizing sequence $\{(\Omega_{n},u_{n}(x,t))\}_{n \in \mN_{*}}$ (Part 2)} \label{Subsection 3.4}
$ $\newline \indent
In this subsection, we will prove that $\{u_{n}(x,t)\}_{n\in\mN_{*}}\subset\mD^{1,2}$ for every $(x,t) \in D\times[0,T]$. In addition, we will show that the Malliavin derivative of $u_{n}(x,t)$ is well-defined as the solution of a random linear integral equation, and that $\{u_{n}\}_{n\in\mN_{*}}\subset\mL^{1,2}$.

Before we proceed, we should make clear the framework of the Malliavin calculus (established in Subsection \ref{Subsection 2.1}) that we are working with. Our case corresponds to the Hilbert space $\mathscr{H} = L^{2}(D \times [0,T])$ and the isonormal Gaussian process $\{\hat{W}(h)\}_{h\in L^{2}(D\times[0,T])}$ given by
\begin{equation} \label{i.G.p. that we work with}
\hat{W}(h):=\int_{0}^{T} \int_{D} h(y,s) W(dy,ds)
\end{equation}
where $W$ is the space-time white noise defined by (\ref{space-time w.n.}); by Remark \ref{R.R. 12}, working with 
$W$ in place of $W_H$ entails no change of filtration. As mentioned in Remark \ref{R.R. 11}, given a random variable in $\mD^{1,2}$, its Malliavin derivative is an element of $L^{2}(\Omega\times D\times[0,T])$.

We are now in a position to prove this subsection's main lemma.
\begin{lemma} \label{main lemma 2}
Let $n\in\mN_{*}$ be fixed and consider the unique solution $u_{n}$ of equation $(\ref{mild formulation of u_n})$. Then: \\[0.05cm]
$1)$ It holds that $u_{n}(x, t) \in \mD^{1,2}$ for every $(x, t) \in D \times [0,T]$. Furthermore, $u_n \in \mL^{1,2}$; in particular, the \phantom{1) }Malliavin derivative $D_{y,s}u_{n}(x, t)$ admits a version which is jointly measurable on $\Omega \times (D\times [0,T])^2$. \\
$2)$ For every $(x, t) \in D \times [0,T]$, up to a $(\mP \otimes dy \otimes ds)$-null subset of $\Omega \times D \times [0,T]$, $D_{y, s}u_n(x, t)$ satisfies
\begin{equation} \label{mal. der. of u_n}
D_{y,s}u_{n}(x,t) = \int_{s}^{t} \int_{D} \Delta_z G(x, z, t - \tau)\, \cg_{n}(z,\tau)\, D_{y,s}u_{n}(z,\tau)\, dzd\tau + \sigma \big{[} K^{*}_{H} \big{(} G(x,\diamond,t-\bullet) \mathbf{1}_{[0,t]}(\bullet) \big{)} \big{]} (y,s)
\end{equation}
on $D \times [0,t]$, where $\cg_{n}(z,\tau) := f_n' \big{(} u_n(z,\tau) \big{)}$, and vanishes on $D \times (t,T]$. Furthermore, uniqueness holds within the class of jointly measurable random fields $V$ vanishing on $D \times (t,T]$ and satisfying
$$\sup_{(x,t) \in D \times [0,T]} E \bigg{(} \int_{0}^{T} \int_{D} \left | V(x,y,t,s\, ;\, n) \right |^{2} dyds \bigg{)} < +\infty.$$
\end{lemma}
\begin{proof}
1) The first step will be to use induction in order to prove some useful properties of the sequence  $\{u_{n,k}\}_{k\in\mN}$ defined in the proof of Lemma \ref{main lemma 1}. Fix now an arbitrary $k \in \mN$. The base case $u_{n,0}(x, t)$ is a deterministic function, so that 
\begin{equation} \label{mal. der. of u_n0}
D_{y,s} u_{n,0}(x,t) = 0,\ \forall (x,t),(y,s) \in D\times[0,T],
\end{equation}
and it trivially satisfies the hypothesis that we now formulate. Our induction hypothesis will consist of three parts, all of them concerning the indices $i \in \mN \cap [0,k]$. Firstly,
$u_{n,i}(x, t) \in \mD^{1,2}$ for every $(x,t) \in D \times [0,T]$. Secondly, $D_{y,s}u_{n,i}(x,t)$ admits a version which is jointly measurable on $\Omega \times (D \times [0, T])^2$; we always work with this version. Lastly, there exists a positive constant $\Lambda = \Lambda (\sigma, d, f, H, n, T)$ (its value is specified at the end of the induction; any positive value suffices for the base case) such that
\begin{equation} \label{induction hyp. 2}
\sup_{t \in [0,T]} \sup_{i \in \mN \cap [0,k]} \sup_{x \in D} E \bigg{(} \int_{0}^{T} \int_{D} |D_{y,s}u_{n,i}(x, t)|^{2}\, dyds \bigg{)}  \leq \Lambda.
\end{equation} 
We note that the integral over $[0,T]$ reduces to one over $[0,t]$, since the Malliavin derivative involved is zero for a.e. $s > t$. 
In order to complete the induction, we have to prove the three requirements for $u_{n, k + 1}$.

To prove the first one, fix an arbitrary $(x,t) \in D \times [0,T]$; we will show that each term on the right-hand side of (\ref{Picard scheme}) belongs to $\mD^{1,2}$. Since $u_{n,0}(x,t) \in \mD^{1,2}$,  only the second and third terms have to be treated. We begin with the second term. Since $u_{n,k}(z,\tau) \in \mD^{1,2}$ for every $(z,\tau) \in D \times [0,T]$ (induction hypothesis) and $f_n \in C^1(\mR)$ with bounded derivative, Proposition 1.2.3 of \cite{N} gives $f_{n} \big{(} u_{n,k}(z,\tau) \big{)} \in \mD^{1,2}$ with 
\begin{equation} \label{mal. der. of composition}
D_{y,s} \left ( f_{n} \big{(} u_{n,k}(z,\tau) \big{)} \right ) = \cg_{n,k}(z,\tau) D_{y,s}u_{n,k}(z,\tau)\, \text{ for } (y,s) \in D \times [0,T]
\end{equation}
$(\mP \otimes dy \otimes ds)$-a.e., where $\cg_{n,k}(z,\tau) := f_n'\big{(} u_{n,k}(z,\tau) \big{)}$; since $|f_n'(z)| \leq c(f,n)$ for every $z \in \mR$, 
\begin{equation} \label{bound of st.p. G}
\sup_{k \in \mN} \sup_{(z,\tau) \in D \times [0,T]} \left | \cg_{n,k}(z,\tau) \right | \leq c(f,n).
\end{equation}
Assume from now on that $t \in (0,T]$ (the case $t = 0$ being trivial) and let $\epsilon \in (0,t]$ be arbitrary. For notational compactness, we define 
\begin{align*}
I_{\epsilon}(x, t) &:= \int_{0}^{t - \epsilon} \int_{D} \Delta_z G(x, z, t - \tau) f_{n} \big{(} u_{n,k}(z, \tau) \big{)}\, dz d\tau, \\
J_{\epsilon}(y, s\, ;\, x, t) &:= \mathbf{1}_{[0, t - \epsilon]}(s) \int_s^{t - \epsilon} \int_D \Delta_z G(x, z, t - \tau) \cg_{n,k}(z, \tau) D_{y,s}u_{n,k}(z, \tau)\, dz d\tau,
\end{align*}
and denote by $I_0(x, t)$, $J_0(y, s\, ;\, x, t)$ the analogous integrals for $\epsilon = 0$.
We will prove that $I_0(x, t) \in \mD^{1,2}$ with
\begin{equation} \label{mal. der. of int. 1}
D_{y,s} \big{(} I_0(x,t) \big{)} = J_0(y,s\,;\, x,t)\, \text{ on } D \times [0,t]\quad \&\quad D_{y,s} \big{(} I_0(x,t) \big{)} = 0\, \text{ on } D \times (t,T],
\end{equation}
$(\mP \otimes dy \otimes ds)$-a.e.
First, we note that $f_{n}(u_{n,k}) \in \mL^{1,2}$. Indeed, by (\ref{mal. der. of composition}), then (\ref{bound of st.p. G}) and Fubini's theorem and finally (\ref{induction hyp. 2}) together with the boundedness of $f_n$, we get
\begin{align*}
\|f_{n}(u_{n,k})\|^{2}_{\mL^{1,2}} 
&= \left \| f_{n} \big{(} u_{n,k}(x,t) \big{)} \right \|_{L^{2}(\Omega\times D\times[0,T])}^{2} + \left \| \cg_{n,k}(x,t) D_{y,s}u_{n,k}(x,t) \right \|^{2}_{L^{2}(\Omega\times (D\times[0,T])^{2})} \\
&\leq c(f, n)^2\, \pi^d\, T\, (1 + \Lambda) < +\infty.
\end{align*}
Second, it holds that $\Delta_\diamond G(x, \diamond, t - \bullet)  \mathbf{1}_{[0, t - \epsilon]}(\bullet) \in L^{2}(D\times[0,T])$, as can be verified using (\ref{Green's function's estimates}) $iii)$ and (\ref{specific calculation}). Thus, by Proposition 3.4.3 of \cite{NE} and Corollary 1.2.1 of \cite{N}, together with (\ref{mal. der. of composition}), we have
$$D_{y,s} \big{(} I_{\epsilon}(x,t) \big{)} = J_{\epsilon}(y,s\,;\, x,t)\, \text{ on } D \times [0,t - \epsilon]\quad \&\quad D_{y,s} \big{(} I_{\epsilon}(x,t) \big{)} = 0\, \text{ on } D \times (t - \epsilon,T],$$
$(\mP \otimes dy \otimes ds)$-a.e. Using the boundedness of $f_n$ together with (\ref{elem. estimate}), we derive
\begin{align*}
|I_0(x,t) - I_\epsilon(x,t)| \leq \int_{t-\epsilon}^t \int_D |\Delta_z G(x, z, t - \tau)|\,  \big{|} f_n\big(u_{n,k}(z,\tau)\big) \big{|}\, dzd\tau &\leq c(d, f, n) \int_{t-\epsilon}^t (t - \tau)^{-\frac{1}{2}}\, d\tau \\
&= c(d, f, n)\, \sqrt{\epsilon}.
\end{align*}
This bound is deterministic, so $I_{\epsilon}(x,t) \to I_0(x,t)$ in $L^p(\Omega)$ as $\epsilon \to 0^+$, for every $p \geq 1$.
Moreover, in the difference $J_0 - J_\epsilon$ the $\tau$-integration is restricted to $[t - \epsilon,t]$, so that Minkowski's integral inequality, (\ref{bound of st.p. G}), (\ref{elem. estimate}) and the induction hypothesis (\ref{induction hyp. 2}) give
$$E \bigg{(} \int_0^t \int_D |J_0(y,s\, ;\, x,t) - J_{\epsilon}(y,s\, ;\, x,t)|^2\, dyds \bigg{)} \leq c(\sigma,d,f,H,n,T)\, \epsilon.$$
Consequently, $J_{\epsilon}(y,s\, ;\, x,t) \to J_0(y,s\, ;\, x,t)$ in $L^2(\Omega\times D \times [0,T])$ as $\epsilon \to 0^+$. The closedness of the operator $D$ now yields $I_0(x,t) \in \mD^{1,2}$ and that (\ref{mal. der. of int. 1}) holds. \\
Regarding now the last term on the right-hand side of (\ref{Picard scheme}), by (\ref{formula for the stoch. int.}) we know that
\begin{equation*}
\sigma \int_{0}^{t}\int_{D} G(x,z,t-\tau) W_{H}(dz,d\tau) = \sigma \int_{0}^{T} \int_{D} \big{[} K^{*}_{H} \big{(} G(x,\diamond,t-\bullet) \mathbf{1}_{[0,t]}(\bullet) \big{)} \big{]} (z,\tau) W(dz,d\tau).
\end{equation*}
The last integral is of the form (\ref{elements of S}) with the single element $h(z,\tau) := \big{[} K^{*}_{H} \big{(} G(x,\diamond,t-\bullet) \mathbf{1}_{[0,t]}(\bullet) \big{)} \big{]} (z,\tau)$, $\varphi := \mathrm{id}_{\mR}$, and the isonormal Gaussian process $\{\hat{W}(h)\}_{h \in L^{2}(D \times [0,T])}$ given by (\ref{i.G.p. that we work with}), thus it belongs to $\cs\subseteq\mD^{1,2}$ with 
\begin{equation} \label{mal. der. of int. 2}
D_{y,s} \bigg{(} \int_{0}^{T} \int_{D} \big{[} K^{*}_{H} \big{(} G(x,\diamond,t-\bullet) \mathbf{1}_{[0,t]}(\bullet) \big{)} \big{]} (z,\tau) W(dz,d\tau) \bigg{)} = \big{[} K^{*}_{H} \big{(} G(x,\diamond,t-\bullet) \mathbf{1}_{[0,t]}(\bullet) \big{)} \big{]} (y,s)
\end{equation}
almost surely, for every $(x,t), (y,s) \in D \times [0,T]$. 

Now that we have shown that every term on the right-hand side of (\ref{Picard scheme}) belongs to $\mD^{1,2}$ for any $(x,t)\in D\times[0,T]$, since $\mD^{1,2}$ is a vector space, this yields the first requirement. Moreover, applying the Malliavin derivative to both sides of (\ref{Picard scheme}) and using the linearity of the operator together with relations (\ref{formula for the stoch. int.}), (\ref{mal. der. of u_n0}), (\ref{mal. der. of int. 1}) and (\ref{mal. der. of int. 2}), we obtain
\begin{align} \label{mal. der. of recursive}
\begin{split}
D_{y,s}u_{n,k+1}(x,t) = \int_{s}^{t} \int_{D} \Delta_z G(x, z, t - \tau) \cg_{n,k}(z,\tau) D_{y,s}u_{n,k}&(z, \tau)\, dzd\tau\ + \\[0.1cm]
&+\sigma \big{[} K^{*}_{H} \big{(} G(x,\diamond,t-\bullet) \mathbf{1}_{[0,t]}(\bullet) \big{)} \big{]} (y,s)
\end{split}
\end{align}
for $(y, s) \in D \times [0, t]$, $(\mP \otimes dy \otimes ds)$-a.e., the derivative vanishing on $D \times (t, T]$. \\
The second requirement is now immediate: as the version of $D_{y, s}u_{n, k + 1}(x, t)$ we choose the right-hand side of (\ref{mal. der. of recursive}) itself, which is jointly measurable on $\Omega \times (D \times [0, T])^2$ by the induction hypothesis and Fubini's
theorem.

We now turn to the third requirement. We take the absolute value, use the triangle inequality and (\ref{bound of st.p. G}), then square, integrate for $(y, s) \in D \times [0, t]$, take expectation and finally take the supremum
over $x \in D$, to obtain
\begin{align} \label{R.R.1}
\begin{split}
\sup_{x \in D} E \bigg{(} \int_{0}^{t} \int_{D} &|D_{y, s}u_{n, k + 1}(x, t) |^{2}\, dyds \bigg{)} \leq \\
&\leq c(f,n)\,\sup_{x \in D} E \bigg{(} \int_{0}^{t} \int_{D} \bigg{(} \int_{s}^{t} \int_{D} \left | \Delta_z G(x, z, t - \tau) \right | \left | D_{y,s}u_{n,k}(z,\tau) \right | dzd\tau \bigg{)}^{2}\, dyds \bigg{)} \\
&\phantom{\leq.}+ c(\sigma) \sup_{x \in D} \int_{0}^{t} \int_{D} \big{|} \big{[} K^{*}_{H} \big{(} G(x, \diamond, t - \bullet) \mathbf{1}_{[0,t]}(\bullet) \big{)} \big{]} (y,s) \big{|}^{2}\, dyds.
\end{split}
\end{align}
Regarding the first term on the right-hand side of (\ref{R.R.1}), Minkowski's integral inequality with respect to the measure $\mP \otimes dy ds$, followed by (\ref{elem. estimate}) and the vanishing of $D_{y, s}u_{n, k}(z, \tau)$ when $s > \tau$, gives
\begin{align*}
E \bigg{(} \int_{0}^{t} \int_{D} \bigg{(} \int_{s}^{t} \int_{D} &|\Delta_z G(x, z, t - \tau)|\, |D_{y, s}u_{n, k}(z, \tau)| dz d\tau \bigg{)}^{2}\, dyds \bigg{)} \leq \\
&\leq c(d) \bigg{(} \int_{0}^{t} (t-\tau)^{-\frac{1}{2}} \bigg{[} \sup_{z \in D} E \bigg{(} \int_{0}^{\tau} \int_{D} |D_{y, s}u_{n, k}(z, \tau)|^{2} dyds \bigg{)} \bigg{]}^{\frac{1}{2}}\, d\tau \bigg{)}^2 \\
\overset{\text{H\"older}}&{\leq} c(d) \bigg{(} \int_0^t (t - \tau)^{-\frac{3}{4}}\, d\tau \bigg{)} \bigg{(} \int_{0}^{t} (t-\tau)^{-\frac{1}{4}} \sup_{z \in D} E \bigg{(} \int_{0}^{\tau} \int_{D} |D_{y, s}u_{n, k}(z, \tau)|^{2} dyds \bigg{)}\, d\tau \bigg{)} \\
&\leq c(d, T) \int_{0}^{t} (t-\tau)^{-\frac{1}{4}} \sup_{z \in D} E \bigg{(} \int_{0}^{\tau} \int_{D} |D_{y, s}u_{n, k}(z, \tau)|^{2} dyds \bigg{)}\, d\tau.
\end{align*}
Since the right-hand side of the last inequality is independent of $x$, combining all the above yields
\begin{align} \label{R.R.2}
\begin{split}
c(f,n)\,\sup_{x \in D} E \bigg{(} \int_{0}^{t} \int_{D} \bigg{(} \int_{s}^{t} &\int_{D} |\Delta_z G(x, z, t - \tau)| \left | D_{y,s}u_{n,k}(z,\tau) \right | dzd\tau \bigg{)}^{2}\, dyds \bigg{)} \leq \\
&\leq c(d, f, n, T) \int_{0}^{t} (t-\tau)^{-\frac{1}{4}} \sup_{z \in D} E \bigg{(} \int_{0}^{\tau} \int_{D} |D_{y, s}u_{n, k}(z, \tau)|^{2} dyds \bigg{)}\, d\tau.
\end{split}
\end{align}
Regarding the second term, (\ref{sup estimate}) with $\zeta=0$ gives
$$\sup_{x \in D} \int_{0}^{t} \int_{D} \big{|} \big{[} K^{*}_{H} \big{(} G(x, \diamond, t - \bullet) \mathbf{1}_{[0,t]}(\bullet) \big{)} \big{]} (y,s) \big{|}^{2}\, dy ds \leq c(d, H)\, t^{2H - \frac{d}{4}},\ \forall t \in [0,T].$$
We take the supremum over $i \in \mN \cap [0,k]$ and define
$$A_{n, k + 1}(t) := \sup_{i \in \mN \cap [0, k + 1]} \sup_{x \in D} E \bigg{(} \int_{0}^{t} \int_{D} |D_{y, s}u_{n, i}(x, t)|^{2}\, dy ds \bigg{)},\ t \in [0, T]$$
(which is finite by the two bounds just obtained and by the induction hypothesis (\ref{induction hyp. 2})). Together with (\ref{mal. der. of u_n0}), the above relations give
\begin{equation} \label{Gronwall for mal. der.}
A_{n, k + 1}(t) \leq c(d, f, n, T) \int_{0}^{t} (t-\tau)^{-\frac{1}{4}} A_{n,k+1}(\tau)\, d\tau + c(\sigma, d, H)\, t^{2H - \frac{d}{4}},\ \forall t \in [0,T].
\end{equation}
Thus, by the singular Gr\"onwall inequality, we get
$$A_{n, k + 1}(t) \leq c(\sigma, d, H)\, \Gamma \Big{(} 2H - \tfrac{d}{4} + 1 \Big{)}\, t^{2H - \frac{d}{4}}\, \text{E}_{\frac{3}{4}, 2H - \frac{d}{4} + 1} \Big{(} c(d, f, n, T)\, \Gamma \big{(} \tfrac{3}{4} \big{)}\, t^{\frac{3}{4}} \Big{)},\ \forall t \in [0,T],$$
where $\Gamma$ is the gamma function and E is the two-parameter Mittag-Leffler function $\text{E}_{a, b}(z) := \sum\limits_{n = 0}^{+\infty} \frac{z^n}{\Gamma(an + b)}$. From the known exponential upper bound of the Mittag-Leffler functions (see \cite[p. 71]{GKMR}), for any $\tilde{c} > 1$, there exists $c > 0$ such that
$$\text{E}_{\frac{3}{4}, 2H - \frac{d}{4} + 1} \Big{(} c(d, f, n, T)\, \Gamma \big{(} \tfrac{3}{4} \big{)}\, t^{\frac{3}{4}} \Big{)} \leq c\, \exp \bigg{[} \tilde{c}\, \Big{(} c(d, f, n, T)\, \Gamma \big{(} \tfrac{3}{4} \big{)}\, t^{\frac{3}{4}} \Big{)}^{\frac{4}{3}} \bigg{]} = c\, \exp \big{(} c(d, f, n, T)\, t \big{)},\ \forall t \in [0,T].$$
Gathering the above we obtain
\begin{align*}
A_{n, k + 1}(t) &\leq c(\sigma, d, H)\, t^{2H - \frac{d}{4}}\, \exp \big{(} c(d, f, n, T)\, t \big{)},\ \forall t \in [0,T] \\
\implies \sup_{t \in [0,T]} A_{n, k + 1}(t) &\leq c(\sigma, d, H)\, T^{2H - \frac{d}{4}} \exp \big{(} c(d, f, n, T)\, T \big{)},
\end{align*}
establishing finally the third requirement of our induction for $\Lambda := c(\sigma, d, H)\, T^{2H - \frac{d}{4}} \exp \big{(} c(d, f, n, T)\, T \big{)}$.

As a consequence of relation (\ref{induction hyp. 2}), for every $(x,t)\in D\times[0,T]$, it holds
\begin{equation} \label{sup bound of mal. der. of u_nk} 
\sup_{k \in \mN} \left \| D_{\diamond, \bullet}u_{n, k}(x, t) \right \|^2_{L^2(\Omega \times D \times [0,T])} \leq \Lambda.
\end{equation}
Since $\{u_{n,k}(x,t)\}_{k \in \mN} \subseteq \mD^{1,2}$ for every $(x,t)$, and we already know (from Lemma \ref{main lemma 1}) that $u_{n,k}(x,t)\to u_{n}(x,t)$ in $L^{2}(\Omega)$, Lemma 1.2.3 of \cite{N} now asserts that $u_{n}(x,t) \in \mD^{1,2}$ for every $(x,t)$ and it holds
\begin{equation} \label{weak conv. of mal. der.}
D_{\diamond, \bullet}u_{n, k}(x, t) \rightharpoonup D_{\diamond, \bullet}u_{n}(x, t) \text{ in } L^{2}(\Omega\times D\times[0,T]) \text{ as}\ k \to +\infty.
\end{equation}
Furthermore, by the second part of the induction hypothesis, $D_{y, s}u_{n, k}$  is jointly measurable and we have
\begin{align*}
\| u_{n, k} \|^{2}_{L^2(\Omega \times D \times [0,T])} + \|D_{y, s}u_{n, k}(x, t) \|^2_{L^2(\Omega \times (D \times [0, T])^2)} \overset{(\ref{induction hyp. 2})}&{\leq} \pi^d\, T \Big{(} \sup_{t \in [0,T]} E \big{(} \|u_{n, k}(\cdot, t) \|_{L^{\infty}(D)}^2 \big{)} + \Lambda \Big{)} \\
\overset{\substack{(\ref{Random bound}) \\ p = 2}}&{\leq} \pi^d\, T \big{[} (c \| u_0 \|_{L^{\infty}(D)} + \cc_2)^2 + \Lambda \big{]}
\end{align*}
uniformly in $k$. Hence $u_{n,k} \in \mL^{1,2}$ for every $k \in \mN$, with $\sup\limits_{k \in \mN} \| u_{n, k} \|_{\mL^{1,2}} < +\infty$. \\ Now, $\mL^{1,2}$ is a Hilbert space, so the sequence $\{u_{n, k}\}_{k \in \mN}$ admits a subsequence converging weakly in $\mL^{1,2}$. Its limit is necessarily $u_n$, because $u_{n, k} \to u_n$ strongly in $L^2(\Omega \times D \times [0,T])$ by Lemma \ref{main lemma 1}. The same applies to every subsequence, whence the whole sequence converges and $u_n \in \mL^{1,2}$; in particular, $D_{y, s}u_n(x,t)$ now has a version that is jointly measurable on $\Omega \times (D \times [0,T])^2$, and we work with it from now on.

2) We first have to identify the integral equation that the Malliavin derivative of $u_n(x, t)$ satisfies. By Corollary 1.2.1 of \cite{N}, it follows immediately that $D_{y,s}u_{n}(x,t) = 0$ for $(y,s) \in D \times (t,T]$, $(\mP \otimes dy \otimes ds)$-a.e., so the focus will be on proving (\ref{mal. der. of u_n}).

Since $f_n \in C^1(\mR)$ with bounded derivative and $u_n(z, \tau) \in \mD^{1,2}$ for every $(z, \tau) \in D \times [0,T]$, Proposition 1.2.3 of \cite{N} gives $f_{n} \big{(} u_n(z,\tau) \big{)} \in \mD^{1,2}$ with 
\begin{equation} \label{mal. der. of composition 2}
D_{y,s} \Big{(} f_{n} \big{(} u_n(z,\tau) \big{)} \Big{)} = \cg_n(z,\tau) D_{y,s}u_n(z,\tau)\, \text{ for } (y,s) \in D \times [0,T],
\end{equation}
$(\mP \otimes dy \otimes ds)$-a.e., where $\cg_n(z,\tau) := f_n'\big{(} u_n(z,\tau) \big{)}$, satisfying 
\begin{equation} \label{bound of st.p. G 2}
\sup_{(z,\tau) \in D \times [0,T]} \left | \cg_n(z,\tau) \right | \leq c(f,n).
\end{equation}
\begingroup
\allowdisplaybreaks
Furthermore, by the weak convergence stated in (\ref{weak conv. of mal. der.}), we get, for every $(x, t) \in D \times [0,T]$
\begin{align*}
E \bigg{(} \int_{0}^{T} \int_{D} |D_{y,s} u_{n}(x, t)|^{2}\, dy ds \bigg{)} &= \|D_{\diamond,\bullet}u_{n}(x, t) \|^{2}_{L^{2}(\Omega \times D \times [0,T])} 
\leq \liminf\limits_{k\to+\infty} \left ( \left \| D_{\diamond,\bullet}u_{n,k}(x, t) \right \|^{2}_{L^{2}(\Omega \times D \times [0,T])} \right ) \\
&\leq \sup_{k\in\mN} \left \| D_{\diamond,\bullet} u_{n,k}(x, t) \right \|^{2}_{L^{2}(\Omega \times D \times [0,T])} \overset{(\ref{sup bound of mal. der. of u_nk})}{\leq} \Lambda
\end{align*}
\endgroup
and therefore
\begin{equation} \label{R.R.4}
\sup_{(x,t)\in D\times[0,T]} E \bigg{(} \int_{0}^{T} \int_{D} \left | D_{y,s} u_{n}(x,t) \right |^{2} dyds \bigg{)} \leq \Lambda.
\end{equation}
Using now (\ref{mal. der. of composition 2}), (\ref{bound of st.p. G 2}), (\ref{R.R.4}), the boundedness of $f_n$ and repeating verbatim the truncation argument used for deriving (\ref{mal. der. of int. 1}), we conclude that
$$\int_0^t \int_D \Delta_z G(x, z, t - \tau)\, f_n \big{(} u_n(z, \tau) \big{)}\, dz d\tau \in \mD^{1,2},\ \forall (x,t) \in D \times [0,T]$$
and, for every such $(x,t)$,
\begin{align} \label{R.R.5}
\begin{split}
D_{y,s} \bigg{(} \int_{0}^{t} \int_{D} \Delta_z G(x, z, t - \tau) f_{n} \big{(} u_{n}(z, \tau) \big{)}\, dz &d\tau \bigg{)} = \int_{s}^{t} \int_{D} \Delta_z G(x, z, t - \tau) \cg_{n}(z, \tau) D_{y,s}u_{n}(z, \tau)\, dz d\tau
\end{split}
\end{align}
for $(y,s) \in D \times [0,t]$, $(\mP \otimes dy \otimes ds)$-a.e., the left-hand side vanishing on $D \times (t,T]$. 

Lastly, we have already explained that the last term on the right-hand side of (\ref{mild formulation of u_n}) belongs to $\mD^{1,2}$ and we have calculated its Malliavin derivative in (\ref{mal. der. of int. 2}). Thus, applying the Malliavin derivative to both sides of (\ref{mild formulation of u_n}) and using (\ref{mal. der. of u_n0}), (\ref{mal. der. of int. 2}) and (\ref{R.R.5}), we finally obtain (\ref{mal. der. of u_n}).

Now that we have shown that the Malliavin derivative of $u_n$ satisfies (\ref{mal. der. of u_n}), we have to justify that it does so uniquely. Let $V$ and $U$ be two jointly measurable random fields satisfying the vanishing and integrability hypotheses of 2) and suppose that both solve the equation obtained from (\ref{mal. der. of u_n}) upon replacing $D_{y,s}u_n(x, t)$ by $V(x,y,t,s\, ;\, n)$ and $U(x,y,t,s\, ;\, n)$ respectively. Upon subtraction of the two equations the inhomogeneous term cancels, so that (\ref{bound of st.p. G 2}) gives
$$|V(x, y, t, s\, ;\, n) - U(x, y, t, s\, ;\, n)| \leq c(f,n) \int_{s}^{t} \int_{D} |\Delta_z G(x, z, t - \tau)|\, |V(z, y, \tau, s\, ;\, n) - U(z, y, \tau, s\, ;\, n)|\, dz d\tau.$$
We then square, integrate over $(y,s)\in D\times[0,t]$, take expectation and finally take the supremum over $x \in D$ to obtain, by the computation leading to (\ref{R.R.2}),
\begin{align*}
\sup_{x\in D} E \bigg{(} \int_{0}^{t} &\int_{D} |V(x, y, t, s\, ;\, n) - U(x, y, t, s\, ;\, n)|^2\, dy ds \bigg{)} \leq \\
&\leq c(d, f, n, T) \int_{0}^{t} (t - \tau)^{-\frac{1}{4}}  \sup_{z \in D} E \bigg{(} \int_{0}^{\tau} \int_{D} |V(z, y, \tau, s\, ;\, n) - U(z, y, \tau, s\, ;\, n)|^2\, dy ds \bigg{)} d\tau.
\end{align*}
If we define now
$$B_{n}(t) := \sup_{x \in D} E \bigg{(} \int_{0}^{t} \int_{D} |V(x,y,t,s\, ;\, n) - U(x,y,t,s\, ;\, n)|^{2}\, dyds \bigg{)},\ t\in[0,T],$$
(which is finite by the integrability assumption imposed on $V$ and $U$) then we have equivalently
$$B_{n}(t) \leq c(d, f, n, T) \int_{0}^{t} (t - \tau)^{-\frac{1}{4}} B_{n}(\tau)\, d\tau,\ \forall t\in[0,T].$$
Thus, by the singular Gr\"onwall inequality and the nonnegativity of $B_{n}$, we obtain $B_{n}(t)=0$ for all $t \in [0,T]$, which in turn yields that $V = U$ almost everywhere on $\Omega \times (D \times [0,T])^2$. 
\end{proof}

In Subsection \ref{Subsection 3.3} we defined a sequence of sets $\{\Omega_{n}\}_{n \in \mN_{*}} \subset \cf$ such that $\Omega_{n} \uparrow \Omega$ a.s. and constructed a sequence of processes $\{u_{n}\}_{n \in \mN_{*}}$ such that $u_{n}(x, t) = u(x, t)$ for all $(x, t) \in D \times [0,T]$ a.s. on $\Omega_{n}$, where $u$ is the unique solution of (\ref{mild formulation}). In this subsection, we proved that $\{u_{n}(x, t)\}_{n \in \mN_{*}} \subset \mD^{1,2}$ for all $(x, t) \in D \times [0,T]$ and showed that its Malliavin derivative is well-defined as the solution of a random linear integral equation. According now to Definition \ref{Def. of local versions}, it holds that $u(x,t) \in \mD^{1,2}_{\textup{\text{loc}}}$ for all $(x, t) \in D \times [0,T]$. Moreover, by 2) of Remark \ref{important remark} we know that the Malliavin derivative of $u$ is defined without ambiguity by $D_{y,s}u(x, t) = D_{y,s}u_{n}(x, t)$ for all $(x, t), (y, s) \in D \times [0,T]$ a.s. on $\Omega_{n},\ n \geq 1$.

We also showed that $\{u_{n}\}_{n \in \mN_{*}} \subset \mL^{1,2}$, which, by 3) of Remark \ref{R.R. 11}, adds to the pointwise membership in $\mD^{1,2}$ the square-integrability of the Malliavin derivative over $\Omega \times (D \times [0,T])^2$. Thus, once again, Definition \ref{Def. of local versions} now implies that $u\in\mL^{1,2}_{\textup{\text{loc}}}$. 
\section{Existence of a density for $u(x,t)$} \label{Section 4}
Consider again the unique solution $u$ of (\ref{mild formulation}). The aim of this section is to prove the existence of a density for the random variable $u(x,t)$, for any $(x,t)\in D\times(0,T]$, by proving the absolute continuity of its law with respect to the Lebesgue measure on $\mR$. 
\subsection{Absolute continuity of the law of $u_{n}(x,t)$ with respect to the Lebesgue measure on $\mR$} \label{Subsection 4.1}
$ $\newline\\[-0.3cm] \indent
Let $n\in\mN_{*}$ be fixed and consider the unique solution $D_{y,s}u_{n}(x,t)$ of (\ref{mal. der. of u_n}).
We will start by proving the following technical lemma:
\begin{lemma} \label{important estimate}
There exist positive constants $c = c(\sigma, d, H)$ and $c' = c'(d, f, n, T)$ such that, for every $\hat{s} \in (0,T]$ and every $\epsilon \in (0, \hat{s}]$, 
\begin{equation} \label{R.R.6} 
\sup_{t\in[\hat{s}-\epsilon,\hat{s}]} \sup_{x \in D} E \bigg{(} \int_{\hat{s}-\epsilon}^{\hat{s}}\int_{D} | D_{y,s}u_{n}(x,t) |^2\, dyds \bigg{)} \leq c\, \epsilon^{2H - \frac{d}{4}} \exp (c'\, \epsilon).
\end{equation}
\end{lemma}
\begin{proof}
We start from equation (\ref{mal. der. of u_n}), take the absolute value on both sides, use the triangle inequality and
(\ref{bound of st.p. G 2}) to get
\begin{align*}
\big{|} D_{y,s}u_{n}(x,t) \big{|} &\leq c(f, n) \int_{s}^{t} \int_{D} |\Delta_z G(x, z, t - \tau)|\, |D_{y,s}u_{n}(z, \tau)|\, dz d\tau + |\sigma| \big{|} \big{[} K^{*}_{H} \big{(} G(x,\diamond,t-\bullet) \mathbf{1}_{[0,t]}(\bullet) \big{)} \big{]} (y,s) \big{|}.
\end{align*}
We then square, integrate for $(y,s) \in D \times [\zeta, t]$, where $0 \leq \zeta \leq t$, take expectation and finally take the supremum over $x \in D$ to obtain
\begin{align} \label{R.R.7}
\begin{split}
\sup_{x \in D} E \bigg{(} \int_{\zeta}^{t} \int_{D} | D_{y,s}&u_{n}(x,t) |^{2}\, dyds \bigg{)} \leq \\
&\leq c(f, n)\, \sup_{x \in D} E \bigg{(} \int_{\zeta}^{t} \int_{D} \bigg{(} \int_{s}^{t} \int_{D} |\Delta_z G(x, z, t - \tau)|\, |D_{y,s}u_{n}(z,\tau)|\, dzd\tau \bigg{)}^{2} dyds \bigg{)} \\
&\phantom{\leq.}+ c(\sigma) \sup_{x \in D} \int_{\zeta}^{t} \int_{D} \left | \big{[} K^{*}_{H} \big{(} G(x,\diamond,t-\bullet) \mathbf{1}_{[0,t]}(\bullet) \big{)} \big{]} (y,s) \right |^{2} dyds.
\end{split}
\end{align}
Regarding the first term on the right-hand side of (\ref{R.R.7}), the computation leading to (\ref{R.R.2}) gives
\begin{align*}
\sup_{x \in D} E \bigg{(} \int_{\zeta}^{t} \int_{D} \bigg{(} \int_{s}^{t} \int_{D} |\Delta_z G(x, z, t - \tau)|\, &|D_{y,s}u_{n}(z, \tau)|\, dz d\tau \bigg{)}^{2} dyds \bigg{)} \leq \\
&\leq c(d, T) \int_{\zeta}^{t} (t-\tau)^{-\frac{1}{4}} \sup_{z \in D} E \bigg{(} \int_{\zeta}^{\tau} \int_{D} | D_{y,s}u_{n}(z,\tau) |^{2}\, dyds \bigg{)} d\tau.
\end{align*}
Regarding now the second term, it is easy to verify that, for any $\zeta\leq s$,
\begin{align*}
\big{[} K^{*}_{H} \big{(} G(x,\diamond,t-\bullet) \mathbf{1}_{[0,t]}(\bullet) \big{)} \big{]} (y,s) = \big{[} K^{*}_{H} \big{(} G(x,\diamond,t-\bullet) \mathbf{1}_{[\zeta,t]}(\bullet) \big{)} \big{]} (y,s)
\end{align*}
and thus
\begin{align*}
\sup_{x \in D} \int_{\zeta}^{t} \int_{D} \big{ |} \big{[} K^{*}_{H} \big{(} G(x,\diamond,t-\bullet) \mathbf{1}_{[0,t]}(\bullet) \big{)} \big{]} (y,s)  \big{|}^{2}\, dyds &\leq \sup_{x \in D} \int_{0}^{T} \int_{D} \left | \big{[} K^{*}_{H} \big{(} G(x,\diamond,t-\bullet) \mathbf{1}_{[\zeta,t]}(\bullet) \big{)} \big{]} (y,s) \right |^{2} dyds \\
\overset{(\ref{sup estimate})}&{\leq} c(d, H)\, (t-\zeta)^{2H - \frac{d}{4}}.
\end{align*}
Gathering the above estimates, we end up with
\begin{align*}
\begin{split}
\sup_{x \in D} E \bigg{(} \int_{\zeta}^{t} &\int_{D} | D_{y,s}u_{n}(x,t) |^{2}\, dyds \bigg{)} \leq \\
&\leq c(d, f, n, T) \int_{\zeta}^{t} (t-\tau)^{-\frac{1}{4}} \sup_{z \in D} E \bigg{(} \int_{\zeta}^{\tau} \int_{D} | D_{y,s}u_{n}(z,\tau) |^{2}\, dyds \bigg{)} d\tau + c(\sigma, d, H)\, (t-\zeta)^{2H - \frac{d}{4}}.
\end{split}
\end{align*}
If we define 
$$L_{n}(t) := \sup_{x \in D} E \bigg{(} \int_{\zeta}^{t} \int_{D} | D_{y,s}u_{n}(x,t) |^{2}\, dy ds \bigg{)},\ t \in [\zeta,T],$$
(which is finite by (\ref{R.R.4})) then the above relation becomes
\begin{equation*}
L_{n}(t) \leq c(d, f, n, T) \int_{\zeta}^{t} (t-\tau)^{-\frac{1}{4}} L_{n}(\tau)\, d\tau + c(\sigma, d, H)\, (t-\zeta)^{2H - \frac{d}{4}},\ \forall t \in [\zeta,T].
\end{equation*}
Thus, by an argument similar to that used for (\ref{Gronwall for mal. der.}), we obtain
$$\sup_{x \in D} E \bigg{(} \int_{\zeta}^{t} \int_{D} | D_{y,s}u_{n}(x,t) |^{2}\, dyds \bigg{)} \leq c(\sigma, d, H)\, (t - \zeta)^{2H - \frac{d}{4}} \exp \big{[} c(d, f, n, T) (t - \zeta) \big{]},\ \forall t \in [\zeta,T].$$
Consider now any $\hat{s} \in (0,T]$, $\zeta\in[0,\hat{s}]$. Then, for every $t \in [\zeta,\hat{s}]$, we have (since $D_{y,s}u_{n}(x,t)=0$ when $s>t$)
$$\sup_{x \in D} E \bigg{(} \int_{\zeta}^{\hat{s}} \int_{D} | D_{y,s}u_{n}(x,t) |^{2}\, dyds \bigg{)} \leq c(\sigma, d, H)\, (\hat{s} - \zeta)^{2H - \frac{d}{4}} \exp \big{[} c(d, f, n, T) (\hat{s} - \zeta) \big{]}.$$
Choosing $\zeta = \hat{s} - \epsilon$ for any $\epsilon \in (0,\hat{s}]$ and taking supremum over $t \in [\hat{s} - \epsilon, \hat{s}]$ yields the desired result.
\end{proof}

We will now use the estimate from the previous lemma to prove the following lemma:
\begin{lemma} \label{abs. con. of law of u_n}
For any $(x,t)\in D\times(0,T]$, the law of the random variable $u_{n}(x,t)$ is absolutely
continuous with respect to the Lebesgue measure on $\mR$.
\end{lemma}
\begin{proof}
In order to prove our claim, we invoke Theorem 2.1.3 of \cite{N} which requires us to show that
$$\Big{(} u_{n}(x,t) \in \mD^{1,1}_{\textup{\text{loc}}}\quad \&\quad \left \| D_{\diamond,\bullet}u_{n}(x,t) \right \|_{L^{2}(D\times[0,T])}>0\ \text{a.s.} \Big{)},\ \forall (x,t)\in D\times(0,T].$$
The first requirement follows immediately from Lemma \ref{main lemma 2} and 1) of Remark \ref{important remark}, so we will focus on the second requirement. 

Starting again from equation (\ref{mal. der. of u_n}), which holds, for every $(x, t) \in D \times [0,T]$, $(\mP \otimes dy \otimes ds)$-a.e. on $\Omega \times D \times [0,t]$, we square and make use of the relation 
$$(A + B)^{2} \geq \dfrac{A^2}{2} - B^2,\ \forall A, B \in \mR\quad \bigg{(} \shs{-0.15}\iff \Big{(} \dfrac{\sqrt{2}}{2} A + \sqrt{2} B \Big{)}^2 \geq 0,\ \forall A, B \in \mR  \bigg{)}$$ 
for
\begin{align*}
A(x, y, t, s) &:= \sigma \big{[} K^{*}_{H} \big{(} G(x,\diamond,t-\bullet) \mathbf{1}_{[0,t]}(\bullet) \big{)} \big{]} (y,s), \\
B(x, y, t, s) &:= \int_{s}^{t} \int_{D} \Delta_z G(x, z, t -\tau) \cg_{n}(z, \tau) D_{y,s}u_{n}(z, \tau)\, dz d\tau,
\end{align*}
to obtain
\begin{equation} \label{R.R.8}
\left | D_{y,s}u_{n}(x,t) \right |^{2} \geq \dfrac{A(x, y, t, s)^2}{2} -  B(x, y, t, s)^2.
\end{equation}
Fix now an arbitrary $(x, t) \in D \times (0, T]$ and consider any $\epsilon \in (0, t)$. We will prove some important estimates that we are going to need later on. More specifically, we will show that
\begin{align}
E \bigg{(} \int_{t-\epsilon}^{t} \int_{D} B(x, y, t, s)^2\, dy ds \bigg{)} &\leq \cc_3\, \epsilon^{2H - \frac{d}{4} + 1}\, \exp (\cc_4\, \epsilon), \label{upper bound} 
\\[0.1cm]
\int_{t - \epsilon}^{t} \int_{D} A(x, y, t, s)^{2}\, dy ds &\geq \cc_5\, \epsilon^{2H} \label{lower bound},
\end{align}
where $\cc_3 = \cc_3(\sigma, d, f, H, n, T),\ \cc_4 = \cc_4(d, f, n, T),\, \cc_5 = \cc_5(\sigma, d, H)$ are positive constants.

Following similar calculations to those of (\ref{R.R.2}), we obtain
\begingroup
\allowdisplaybreaks
\begin{align*}
E \bigg{(} \int_{t-\epsilon}^{t} \int_{D} B(x, y, t, s)^2\, dy ds &\bigg{)} \leq c(d, f, n) \bigg{(} \int_{t - \epsilon}^{t} (t-\tau)^{-\frac{1}{2}} \bigg{[} \sup_{z \in D} E \bigg{(} \int_{t - \epsilon}^{\tau} \int_{D} |D_{y, s}u_n(z, \tau)|^{2} dyds \bigg{)} \bigg{]}^{\frac{1}{2}}\, d\tau \bigg{)}^2 \\
&\leq c(d, f, n) \sup_{\tau\in[t - \epsilon,t]} \sup_{z \in D} E \bigg{(} \int_{t - \epsilon}^t \int_D \left | D_{y,s}u_{n}(z,\tau) \right |^2 dyds \bigg{)} \bigg{(} \int_{t -\epsilon}^{t} (t - \tau)^{-\frac{1}{2}} d\tau \bigg{)}^2 
\\[0.1cm]
\overset{(\ref{R.R.6})}&{\leq} c(\sigma, d, f, H, n, T)\, \epsilon^{2H - \frac{d}{4}} \exp \big{(} c(d, f, n, T)\, \epsilon \big{)} \bigg{(} \int_{t -\epsilon}^{t} (t - \tau)^{-\frac{1}{2}} d\tau \bigg{)}^2 \\[0.1cm]
&= c(\sigma, d, f, H, n, T)\, \epsilon^{2H - \frac{d}{4} + 1} \exp \big{(} c(d, f, n, T)\, \epsilon \big{)},
\end{align*}
\endgroup
which establishes (\ref{upper bound}). 
We now turn to the proof of (\ref{lower bound}):
\begin{align*}
\int_{t - \epsilon}^{t} \int_{D} A(x, y, t, s)^2\, dy ds
&= \sigma^{2} \int_{t - \epsilon}^{t} \int_{D} \bigg{(} \int_{s}^{t} G(x, y, t - r)\, 
\frac{\partial K_H}{\partial r}(r,s)\, dr \bigg{)}^{2}\, dy ds \\[0.3cm]
&= \sigma^{2} \int_{t - \epsilon}^{t} \int_{s}^{t} \int_{s}^{t} \int_{D}
\frac{\partial K_H}{\partial r}(r,s)\, \frac{\partial K_H}{\partial r'}(r',s)\, 
G(x, y, t - r)\, G(x, y, t - r')\, dy dr dr' ds \\[0.3cm]
&\overset{(\ref{int. property of Green's function})}{=} \sigma^{2}\int_{t-\epsilon}^{t}\!\int_{s}^{t}\!\int_{s}^{t}
\frac{\partial K_H}{\partial r}(r,s)\,\frac{\partial K_H}{\partial r'}(r',s)\,
G(x,x,2t-r-r')\,dr\,dr'\,ds .
\end{align*}
The set $\{r = t\} \cup \{r' = t\}$ is $dr\,dr'$-null and may therefore be discarded, so that $2t - r - r' > 0$ throughout and the corresponding series of $G(x, x, 2t - r - r')$
converges; by the lower bound of (\ref{diagonal bounds}), which holds at every positive time, we have $G(x, x, 2t - r - r') \geq \pi^{-d}$. Since the two derivatives are positive by (\ref{par. der. of K_H}), this bound may be inserted; the remaining integrand then depends on $r$ and $r'$ separately, so that
\begin{align*}
\sigma^{2}\int_{t-\epsilon}^{t}\!\int_{s}^{t}\!\int_{s}^{t}
\frac{\partial K_H}{\partial r}(r,s)\, \frac{\partial K_H}{\partial r'}(r',s)\,
G(x, x, 2t - r - r')\, dr dr' ds &\geq \pi^{-d} \sigma^{2} \int_{t-\epsilon}^{t}
\bigg{(} \int_s^t \frac{\partial K_H}{\partial r}(r,s)\, dr \bigg{)}^{2} ds \\
\overset{(\ref{int. of par. der. of K_H})}&{=} \pi^{-d} \sigma^{2} \int_{t - \epsilon}^{t} K_H(t,s)^{2}\, ds.
\end{align*}
It remains to bound $K_H(t,s)$ from below. By (\ref{S.1}) and the monotonicity of
$u \mapsto u^{H-\frac{1}{2}}$,
$$K_H(t,s) = \kappa_1\, s^{\frac{1}{2} - H} \int_s^t (u-s)^{H - \frac{3}{2}}\,
u^{H - \frac{1}{2}}\, du
\ \geq\ \kappa_1\, s^{\frac{1}{2} - H}\, s^{H - \frac{1}{2}} \int_s^t (u - s)^{H - \frac{3}{2}}\, du = \frac{\kappa_1}{H - \frac{1}{2}}\, (t - s)^{H - \frac{1}{2}},$$
thus
$$\int_{t - \epsilon}^t K_H(t,s)^2\, ds \geq \frac{\kappa_1^2}{\big(H - \frac{1}{2}\big)^2} \int_{t - \epsilon}^t (t - s)^{2H - 1}\, ds = \frac{\kappa_1^2}{2H \big(H - \frac{1}{2}\big)^2}\, \epsilon^{2H}.$$
Combining all the above we obtain (\ref{lower bound}).

We are now in a position to conclude the proof of the lemma. We have
\begingroup
\allowdisplaybreaks
\begin{align*}
1 &\geq \mP \bigg{(} \int_{0}^{T} \int_{D} |D_{y,s}u_{n}(x, t)|^{2}\, dy ds > 0\bigg{)} \geq \mP \bigg{(} \int_{t-\epsilon}^{t} \int_{D} |D_{y,s}u_{n}(x, t)|^{2}\, dy ds > 0 \bigg{)} \\[0.1cm]
\overset{(\ref{R.R.8})}&{\geq} \mP \bigg{(} \dfrac{1}{2} \int_{t-\epsilon}^{t} \int_{D} A(x, y, t, s)^2\, dy ds -\int_{t-\epsilon}^{t} \int_{D} B(x, y, t, s)^2\, dy ds > 0\bigg{)} 
\\[0.1cm]
\overset{(\ref{lower bound})}&{\geq} \mP \bigg{(} \int_{t-\epsilon}^{t} \int_{D} B(x, y, t, s)^2\, dy ds < \frac{\cc_5\, \epsilon^{2H}}{2} \bigg{)} = 1- \mP \bigg{(} \int_{t - \epsilon}^{t} \int_{D} B(x, y, t, s)^2\, dyds 
\geq \frac{\cc_5\, \epsilon^{2H}}{2} \bigg{)} \\[0.1cm]
\overset{\text{Markov}}&{\geq} 1 - \frac{2 E \Big{(} \int_{t-\epsilon}^{t} \int_{D} B(x, y, t, s)^2\, dyds \Big{)}}{\cc_5\, \epsilon^{2H}} \overset{(\ref{upper bound})}{\geq} 1- \frac{2\cc_3\, \epsilon^{2H - \frac{d}{4} + 1}\, \exp (\cc_4\, \epsilon)}{\cc_5\, \epsilon^{2H}} \\[0.1cm]
&= 1 - \frac{2 \cc_3}{\cc_5}\, \epsilon^{1 - \frac{d}{4}}\, \exp (\cc_4\, \epsilon).
\end{align*}
\endgroup
Since $1 - \frac{d}{4} > 0$ for every $d \in \{1,2,3\}$, taking limits as $\epsilon \to 0^{+}$  yields
$$\mP \bigg{(} \int_{0}^{T} \int_{D} |D_{y,s}u_{n}(x, t)|^{2}\, dy ds > 0 \bigg{)} = 1 \iff \left \| D_{\diamond,\bullet}u_{n}(x, t) \right \|_{L^{2}(D\times[0,T])} > 0\ \text{a.s.}$$
for any $(x,t)\in D\times(0,T]$, which is exactly the second requirement mentioned at the beginning.
\end{proof}
\subsection{Absolute continuity of the law of $u(x,t)$ with respect to the Lebesgue measure on $\mR$} \label{Subsection 4.2}
$ $\newline\\[-0.3cm] \indent
Consider one last time the unique solution $u$ of (\ref{mild formulation}). Everything we have studied so far leads to the following lemma:
\begin{lemma} \label{abs. con. of law of u}
For any $(x,t)\in D\times(0,T]$, the law of the random variable $u(x,t)$ is absolutely continuous with respect to the Lebesgue measure on $\mR$.
\end{lemma}
\begin{proof}
As we did for the analogous result in Lemma \ref{abs. con. of law of u_n}, we will invoke Theorem 2.1.3 of \cite{N} which requires us to show that
$$\Big{(} u(x,t) \in \mD^{1,1}_{\textup{\text{loc}}}\quad \&\quad \left \| D_{\diamond,\bullet}u(x,t) \right \|_{L^{2}(D\times[0,T])}>0\ \text{a.s.} \Big{)},\ \forall (x,t)\in D\times(0,T].$$

The first requirement is again immediate by the argument at the end of Section \ref{Section 3} and 1) of Remark \ref{important remark}.

Regarding the second requirement, we deduce it from the corresponding property of $u_{n}(x,t)$, established in Lemma \ref{abs. con. of law of u_n}, rather than proving it directly. In detail, we fix an arbitrary $n \in \mN_{*}$ and define the sets
\begin{alignat*}{2}
A_{1} &:= \Big{\{} \omega\in\Omega\, :\, \left \| D_{\diamond,\bullet}u(x,t;\omega) \right \|_{L^{2}(D\times[0,T])}>0 \Big{\}},\quad &&B_{1} := \Big{\{} \omega\in\Omega\, :\, \left \| D_{\diamond,\bullet}u_{n}(x,t;\omega) \right \|_{L^{2}(D\times[0,T])}>0 \Big{\}}, \\[0.1cm]
A_{2} &:= \Big{\{} \omega\in\Omega_{n}\, :\, \left \| D_{\diamond,\bullet}u(x,t;\omega) \right \|_{L^{2}(D\times[0,T])}>0 \Big{\}},\quad &&B_{2} := \Big{\{} \omega\in\Omega_{n}\, :\, \left \| D_{\diamond,\bullet}u_{n}(x,t;\omega) \right \|_{L^{2}(D\times[0,T])}>0 \Big{\}},
\end{alignat*}
where $\Omega_{n}$ is given by (\ref{def. of Omega_n}). We explained at the end of Section \ref{Section 3} that $D_{y,s}u(x,t) = D_{y,s}u_{n}(x,t)$ for all $(x,t),(y,s) \in D \times [0,T]$, a.s. on $\Omega_{n}$; therefore it holds that
$\mP(B_{2}) = \mP(A_{2})$.\\
Additionally, it was established in Lemma \ref{abs. con. of law of u_n} that $\mP(B_{1})=1$; thus $\mP(B_{1}^{c})=0$. Notice, however, that
$$\Omega_{n} - B_{2} = \left \{ \omega\in\Omega_{n}\, :\, \left \| D_{\diamond,\bullet}u_{n}(x,t;\omega) \right \|_{L^{2}(D\times[0,T])}=0 \right \} \subseteq \left \{ \omega\in\Omega\, :\, \left \| D_{\diamond,\bullet}u_{n}(x,t;\omega) \right \|_{L^{2}(D\times[0,T])}=0 \right \} = B_{1}^{c}$$
so we get $\mP(\Omega_{n} - B_{2}) \leq \mP(B_{1}^{c})$, which yields
$$\mP(\Omega_{n} - B_{2}) = 0 \iff \mP(\Omega_{n}) = \mP(B_{2}) \implies \mP(\Omega_{n}) = \mP(A_{2}) \overset{A_{2} \subseteq A_{1}}{\implies} \mP(\Omega_{n}) \leq \mP(A_{1}).$$
The last inequality holds for every $n \in \mN_{*}$ so, passing to the limit as $n \to +\infty$, we obtain $\mP(A_{1})=1$ which is exactly the second requirement.
\end{proof}
\begin{proof}[Proof of Theorem \ref{main theorem}]
The first assertion was established at the end of Section \ref{Section 3}: $\{(\Omega_{n}, u_{n}(x,t))\}_{n \in \mN_{*}}$ localizes $u(x, t)$ in $\mD^{1,2}$, so that $u(x, t) \in \mD^{1,2}_{\textup{\text{loc}}}$ for every $(x, t) \in D \times [0,T]$, with $D_{y, s}u(x, t)$ defined without ambiguity by 2) of Remark \ref{important remark}. The second assertion is Lemma \ref{abs. con. of law of u}, which in turn yields the existence of a density for $u(x,t)$ at every $(x, t) \in D \times (0,T]$.
\end{proof}
\section*{Acknowledgments}
We are grateful to the anonymous referee for the valuable comments provided. The research work was implemented in the framework of the H.F.R.I. call ``Basic Research Financing (Horizontal support for all Sciences)'' under the National Recovery and Resilience Plan ``Greece 2.0'' funded by the European Union - NextGenerationEU. (H.F.R.I. Project Number: 14910).

\bibliographystyle{plain}

\begin{thebibliography}{99}
\bibitem{AFK1}
D.C. Antonopoulou, D. Farazakis, G. Karali, Malliavin Calculus for the stochastic Cahn-Hilliard/Allen-Cahn equation with unbounded noise diffusion, Journal of Differential Equations, 265(7), pp. 3168–3211, (2018).

\bibitem{AKM}
D.C.  Antonopoulou, G. Karali, A. Millet, Existence and regularity of solution for a stochastic Cahn-Hilliard/Allen-Cahn equation with unbounded noise diffusion, J. Differential Equations 260, 2383-2417, (2016).

\bibitem{BaTu}
R.M. Balan, C.A. Tudor, The stochastic heat equation with fractional-colored noise: existence of the solution. ALEA Lat. Am. J. Probab. Math. Stat. 4, 57–87, (2008).

\bibitem{BJW}
L. Bo, Y. Jiang, Y. Wang, Stochastic Cahn-Hilliard equation with fractional noise, Stoch. Dyn. 8 (2008), 643–665.

\bibitem{CH}
J. W. Cahn, J. E. Hilliard, Free energy of a nonuniform system. III. Nucleation in a two-component incompressible
fluid, J. Chem. Phys., 31, 688-699, (1959).

\bibitem{WEB1}
C. Cardon-Weber, Cahn-Hilliard stochastic equation: Existence of the solution and of its density.
Bernoulli 7 (5), pp. 777-816, (2001).

\bibitem{C}
H. Cook, Brownian motion in spinodal decomposition, Acta Metallurgica, 18, 297-306, (1970).

\bibitem{dalang}
R.C. Dalang, M. Sanz-Solé, Stochastic Partial Differential Equations, Space-time White Noise and Random Fields,  arXiv preprint 2402.02119, (2024).

\bibitem{G}
A.M. Garsia, Continuity properties of Gaussian processes with multidimensional time parameter, in: Proceedings of the Sixth Berkeley Symposium on Mathematical Statistics and Probability, Vol. 2, University of California Press, Berkeley, pp. 369-374, (1972).

\bibitem{GRR}
A.M. Garsia, E. Rodemich, H. Rumsey Jr., A real variable lemma and the continuity of paths of some Gaussian processes, Indiana University Mathematics Journal, 20(6), 565-578, (1970). 

\bibitem{GKMR}
R. Gorenflo, A. A. Kilbas, F. Mainardi, S. Rogosin, Mittag-Leffler Functions, Related Topics and Applications, 2nd ed., Springer Monographs in Mathematics, Springer, Berlin, Heidelberg, (2020).

\bibitem{He}
D. Henry, Geometric Theory of Semilinear Parabolic Equations, Lecture Notes in Mathematics, Vol. 840, Springer-Verlag, Berlin-New York, (1981).

\bibitem{H}
Y. Hu, Heat equations with fractional white noise potentials, Appl. Math. Optim. 43, 221-243, (2001).

\bibitem{Kun}
H. Kunita, Stochastic Flows and Stochastic Differential Equations, Cambridge Studies in Advanced Mathematics, Vol. 24, Cambridge University Press, Cambridge, (1990).

\bibitem{LTWW}
W. Leland, M. Taqqu, W. Willinger, D. Wilson, On the self-similar nature of ethernet traffic
(Extended version), IEEE/ACM Trans. Networking 2 (1994) 1–15.

\bibitem{M}
B.B. Mandelbrot, The variation of certain speculative prices, J. Business 36 (1963) 394–419.

\bibitem{N}
D. Nualart, The Malliavin Calculus and Related Topics, 2nd ed., Springer, Berlin, 2006.

\bibitem{NE}
D. Nualart, E. Nualart. Introduction to Malliavin Calculus. Institute of Mathematical Statistics Textbooks. Cambridge University Press, Cambridge, 2018.

\bibitem{NO}
D. Nualart and Y. Ouknine, Regularization of quasilinear heat equation by a fractional noise, Stoch. Dyn. 4 (2004) 201–221.

\bibitem{NS}
D. Nualart and B. Saussereau, Malliavin calculus for stochastic differential equations driven by a fractional Brownian motion, Stochastic Process. Appl. 119(2009), 391–409.

\bibitem{W}
J.B. Walsh, An introduction to stochastic partial differential equations, Lecture Notes in Mathematics, vol. 1180, Springer, Berlin, 1986, pp. 265-439.

\end{thebibliography}

\end{document}